\documentclass[reqno,11pt]{amsart}
\usepackage{amsmath,amssymb}
\usepackage{paralist}
\usepackage{graphics} 
\usepackage{epsfig} 
\usepackage[colorlinks=true]{hyperref}
\hypersetup{urlcolor=red, citecolor=blue}
\usepackage[top=1in,bottom=1in,left=1in,right=1in]{geometry}
\usepackage[german,french,english]{babel}
\usepackage{amsmath,amssymb}

\newtheorem{thm}{Theorem}[section]

\newtheorem{lem}[thm]{Lemma}

\theoremstyle{definition}
\newtheorem{defn}{Definition}[section]
\newtheorem{rem}{Remark}[section]
\numberwithin{equation}{section}

\begin{document}
\title[Stochastic wave equations with nonlinear damping and source terms]
      {stochastic wave equations with nonlinear damping and source terms}%
\author[Gao, H.-J.]{Gao Hongjun$^{1, 3, *}$}
\author[Guo, B-L.]{Guo Boling$^4$}
\author[Liang, F.]{Liang Fei$^{1, 2}$}%

\address{1. Jiangsu Provincial Key Laboratory for Numerical Simulation of
Large Scale Complex Systems\\ School of Mathematical Science,
Nanjing Normal University, Nanjing 210046, PR China}

\address{2. Department of Mathematics, An Hui Science And Technology
University, Feng Yang, 233100, Anhui, PR China}

\address{3. Center of Nonlinear Science\\ Nanjing University, Nanjing
210093, China\\}

\address{4. Institute of Applied Physics and
Computational
Mathematics\\ P. O. Box 8009, Beijing 100088, China\\}


\thanks{Supported in part by a China NSF Grant No. 10871097,
No. 11028102 and National Basic Research Program of China (973
Program) No. 2007CB814800.}
\thanks{*Corresponding author.}
\thanks{E-mail: gaohj@hotmail.com, gaohj@njnu.edu.cn(H.Gao),\ fliangmath@126.com}


\begin{abstract}
In this paper, we discuss an initial boundary value problem for
the stochastic wave equation involving the nonlinear damping term
$|u_t|^{q-2}u_t$ and a source term of the type $|u|^{p-2}u$. We
firstly establish the local existence and uniqueness of solution
by the Galerkin approximation method and show that the solution is
global for $q\geq p$. Secondly, by an appropriate energy
inequality, the local solution of the stochastic equations will
blow up with positive probability or explosive in energy sense for
$p>q$.

\end{abstract}
\maketitle

\section{Introduction}

The wave equation of the following form
\begin{eqnarray}\label{main}
  \begin{cases}
     u_{tt}-\Delta u+a|u_t|^{q-2}u_t=b|u|^{p-2}u,\ \ \ &(x,t)\in D \times (0,T),\\
    u(x,t)=0,            &(x,t)\in\partial D\times (0,T),\\
    u(x,0)=u_{0}(x),\ \ \ u_t(x,0)=u_1(x),\ \ \ &x\in D,
\end{cases}
\end{eqnarray}
where $D$ is a bounded domain in $\mathbb{R}^d$ with a smooth
boundary $\partial D$, $a,\ b>0$ are constants, has been
extensively studied and results concerning existence, blow-up and
asymptotic behavior of smooth, as well as weak solutions have been
established by several authors over the past three decades. For
$b=0$, it is well known that the damping term assures global
existence and decay of the solution energy for arbitrary initial
data (see \cite{HZ} and \cite{K}). For $a=0$, the source term
causes finite time blow-up of solutions with large initial data
(negative initial energy), see \cite{B} and \cite{KL}. The
interaction between the damping term $a|u_t|^{q-2}u_t$ and the
source term $b|u|^{p-2}u$ makes the problem more interesting. This
situation was first considered by Levine \cite{L,L1} in the linear
damping case ($q=2$), where he showed that solutions with negative
initial energy blow up in finite time. In \cite{GT}, Georgiev and
Todorova extended Levine's result to the nonlinear damping case
($q>2$). In their work, the authors introduced a new method and
determined relations between $q$ and $p$ for which there is finite
time blow-up. Specifically, they showed that solutions with
negative energy continue to exist globally in time if $q\geq p\geq
2$ and blow up in finite time if $p>q \geq 2$ and the initial
energy is sufficiently negative. Messaoudi \cite{M1} extended the
blow-up result of \cite{GT} to solutions with only negative
initial energy. For related results, we refer the reader to Levine
and Serrin \cite{LS}, Levine and Ro Park \cite{LR}, Vitillaro
\cite{V} and Messaoudi and Said-Houari \cite{MS}.

In fact, the driving force may be affected by the environment
randomly. In view of this, we consider the following stochastic
wave equations
\begin{eqnarray}\label{smain}
  \begin{cases}
     u_{tt}-\Delta u+|u_t|^{q-2}u_t=|u|^{p-2}u+\varepsilon
     \sigma (u,\nabla u,x,t)\partial_t W(t,x),\ \ \ &(x,t)\in D \times (0,T),\\
    u(x,t)=0,            &(x,t)\in\partial D\times (0,T),\\
    u(x,0)=u_{0}(x),\ \ \ u_t(x,0)=u_1(x),\ \ \ &x\in D,
\end{cases}
\end{eqnarray}
where $q\geq2,\ p>2$, $\varepsilon$ is given positive constant
which measures the strength of the noise, and $W(t,x)$ is a Wiener
random field, which will be defined precisely later, and the
initial data $u_0(x)$ and $u_1(x)$ are given functions.

To motivate our work, let us recall some results regarding
stochastic wave equations with linear damping ($q=2$). For the
blow-up results, Chow \cite{C} discussed a class of
non-dissipative stochastic wave equations with polynomial
nonlinearity in $\mathbb{R}^d$ with $d\leq 3$. Using the energy
inequality the author demonstrated the blow-up in finite time with
a positive probability or explosive in $L^2$ norm for an example
and studied the global existence of the solutions for the
equation. This blow-up result has been later generalized by the
same author in \cite{C3}. In a recent paper, using the energy
inequality, Bo et al. \cite{BTW} proposed sufficient conditions
that the solutions of a class of stochastic wave equations blow up
with a positive probability or explosive in $L^2$ sense. In those
papers, the main tool in proving explosive/blow-up is the
``concavity method" where the basic idea of the method is to
construct a positive defined functional $F(t)$ of the solution by
the energy inequality and show that $F^{-\alpha}(t)$ is a concave
function of $t$. Unfortunately, this method fails in the case of a
nonlinear damping term ($q>2$). For the global existence and
invariant measure, Chow \cite{C1,C2} studied properties of the
solution of (\ref{smain}) with $q=2$ such as asymptotic stability
and invariant measure and Brze$\acute{z}$niak et al. \cite{BM}
studied global existence and stability of solutions for the
stochastic nonlinear beam equations. There are also many other
works on the stochastic wave equations with global existence and
invariant measure for linear damping, see references in \cite{CN,
HC, DF, MM}.

Nonlinear stochastic wave equations with nonlinearity on the
damping were first studied by Pardoux \cite{EP}. But the progress
is little in nearly three decades. Recently, J.U. Kim \cite{JUK}
and  V. Barbu et al. \cite{BPT} considered an initial boundary
value stochastic wave equations with nonlinear damping and
dissipative damping, respectively. They proved the existence of an
invariant measure. However, to our knowledge, the
explosive/blow-up results with nonlinearity on the damping seems
to be studied here for the first time. Since the existence and
uniqueness of a solution for the deterministic equation
($\varepsilon=0$) is well known under some assumptions with
nonlinearity on the damping, we may anticipate similar results for
the stochastic equation. However, the methods used in earlier
works on the stochastic wave equation with linear damping do not
work. Hence, we will employ the Galerkin approximation method to
establish the local existence and uniqueness solution for
(\ref{smain}). For multiplicative noise, i.e., when $\sigma$
depend on $u$ and $\nabla u$, we need to obtain the mean energy
estimates, but this is some technical difficulty. This is also a
major hurdle for the uniqueness of a solution of (\ref{smain}). So
here we consider only additive noise, i.e. $\sigma(u,\nabla u,
x,t)=\sigma(t,x,\omega)$ so that the stochastic integral may be
well defined as an $L^2(D)$-valued continuous martingale. We will
prove the global solution of (\ref{smain}) for $q\geq p$.
Concerning explosive/blow-up results, we use the technique of
\cite{GT} with a modification in the energy functional due to the
different nature of the problems for $p>q$.

This paper is organized as follows. In Section $2$ we present some
assumptions and definitions needed for our work. In Section $3$,
we show the local existence and uniqueness solution of
(\ref{smain}) and prove the solution being global for $q\geq p$.
Section $4$ is devoted to the proof of the explosive solutions of
(\ref{smain}) for $p>q$.

\section{Preliminaries}

Firstly, let us introduce some notation used throughout this
paper. We set $H=L^2(D)$ with the inner product and norm denoted
by $(\cdot,\cdot)$ and $||\cdot||_2$, respectively. Denote by
$||\cdot||_q$ the $L^q(D)$ norm for $1\leq q\leq\infty$ and by
$||\nabla\cdot||_2$ the Dirichlet norm in $V=H^1_0(D)$ which is
equivalent to the $H^1(D)$ norm. We also set $q,\ p$ satisfy
\begin{eqnarray}\label{condition}
\begin{cases}
q\geq2,\ \ \ p>2,\ \ \ \max\{p,\ q\}\leq \displaystyle\frac{2(d-1)}{d-2}, \ \ \ &{\rm if}\ d\geq3,\\
 q\geq2,\ \ \ p>2,&{\rm if} \
       d=1,2,
\end{cases}
\end{eqnarray}
which implies that $H_0^1(D)$ is continuously compact embedded into
$ L^p(D)$. Hence, we have the Sobolev inequality
\begin{equation}\label{sobolev}
||u||_{2(p-1)}\leq c||\nabla u||_2,\ \ \ \forall u\in H_0^1(D),
\end{equation}
where $c$ is the embedding constant of $H_0^1(D)\subseteq L^p(D)$.
Using (\ref{sobolev}), we have the following inequality
\begin{equation}\label{sobolev1}
||u^{p-2}v||_2\leq c^{p-1}||\nabla u||_2^{p-2}||\nabla v||_2, \ \ \
\forall u,\ v\in H_0^1(D).
\end{equation}
In fact, when $d=1,\ 2$, let $q>1$ and $k=\frac{q}{q-1}$, by the
H\"{o}lder inequality and (\ref{sobolev}) we have
\begin{equation}\label{sobolev2}
||u^{p-2}v||_2\leq ||u||^{p-2}_{2(p-2)q}||v||_{2k}\leq c^{p-1}||\nabla
u||_2^{p-2}||\nabla v||_2.
\end{equation}
When $d>2$, set $q=\frac{d}{(d-2)(p-2)}>1$. Then
$k=\frac{d}{d-(d-2)(p-2)}\leq \frac{d}{d-2}$, (\ref{sobolev2}) is also
valid for $d>2$.

Let $(\Omega,P,\mathcal{F})$ be a complete probability space for
which a $\{\mathcal{F}_t,\ t\geq0\}$ of sub-$\sigma$-fields of
$\mathcal{F}$ is given. A point of $\Omega$ will be denoted by
$\omega$ and $\textbf{E}(\cdot)$ stands for expectation with
respect to probability measure $P$. When $\mathcal{O}$ is a
topological space, $\mathcal{B}$ denotes the Borel
$\sigma$-algebra over $\mathcal{O}$. Suppose that
$\{W(t,x):t\geq0\}$ is a $V$-valued $R$-Wiener process on the
probability space with the variance operator $R$ satisfying
$TrR<\infty$. Moreover, we can assume that $R$ has the following
form
\[
 Re_i=\lambda_ie_i,\ \ \ i=1,2,\cdots,
\]
where ${\lambda_i}$ are eigenvalues of $R$ satisfying
$\sum_{i=1}^\infty\lambda_i<\infty$ and $\{e_i\}$ are the
corresponding eigenfunctions with
$c_0:=\sup_{i\geq1}||e_i||_\infty<\infty$ (where
$||\cdot||_\infty$ denotes the super-norm). To simplify the
computations, we assume that the covariance operator $R$ and
$-\Delta$ with homogeneous Dirichlet boundary condition have a
common set of eigenfunctions, i.e., $\{e_i\}_{i=1}^\infty$ satisfy
\begin{eqnarray}\label{ef}
  \begin{cases}
     -\Delta e_i=\mu_ie_i,\ \ \ &x\in D,\\
     e_i=0,\ \ \ &x\in\partial D,
\end{cases}
\end{eqnarray}
and  form an orthonormal base of $V$. In this case,
\[
 W(t,x)=\sum_{i=1}^\infty\sqrt{\lambda_i}B_i(t)e_i,
\]
where $\{B_i(t)\}$ is a sequence of independent copies of standard
Brownian motions in one dimension. Let $\mathcal{H}$ be the set of
$L^0_2=L^2(R^{\frac{1}{2}}V,V)$-valued processes with the norm
\[
 ||\Psi(t)||_{\mathcal{H}}=\Big(\textbf{E}\int_0^t||\Psi(s)||_{L^0_2}^2ds\Big)^{\frac{1}{2}}=
 \Big(\textbf{E}\int_0^tTr\big(\Psi(s)R
 \Psi^*(s)\big)ds\Big)^{\frac{1}{2}}<\infty,
\]
where $\Psi^*(s)$ denotes the adjoint operator of $\Psi(s)$. Let
$\{t_k\}_{k=1}^n$ be a partition on $[0,T]$ such that
$0=t_0<t_1<\cdots<t_n=T$. For a process $\Psi(t)\in \mathcal{H}$,
define the stochastic integral with respect to the $R$-Wiener
process as
\begin{equation}\label{condition2}
 \int_0^t
 \Psi(s)dW(s)=\lim_{n\rightarrow\infty}\sum_{k=0}^{n-1}\Psi(t_k)(W(t_{k+1}\wedge t)-W(t_k\wedge
 t)),
\end{equation}
where the sequence converges in $\mathcal{H}$-sense. It is not
difficult to check that the integral process $\int_0^t
 \Psi(s)dW(s)$  is a martingale for any $\Psi(t)\in \mathcal{H}$, and the quadratic variation process is given by
\[
\Big\langle\Big\langle\int_0^t
 \Psi(s)dW(s)\Big\rangle\Big\rangle=\int_0^t
 Tr\big(\Psi(s)R
 \Psi^*(s)\big)ds.
\]
For more details about the infinite dimension Wiener process and
the stochastic integral, we refer to \cite{PZ}.

Finally, we give the definition of solution to (\ref{smain}). For
the definition of a solution, we assume that
\begin{equation}\label{mild}
 (u_0,u_1)\in
H_0^1(D)\times L^2(D),
\end{equation}
and that $\sigma(x,t)$ is $L^2(D)$-valued progressively measurable
such that
\begin{equation}\label{mild1}
  \textbf{E}\int_0^T
||\sigma(t)||_2^2dt<\infty,
\end{equation}
\begin{defn}\label{defn}
Under the assumption (\ref{mild}) and (\ref{mild1}), $u$ is said
to be a solution of (\ref{smain}) on the interval $[0,T]$ if
\begin{equation}\label{define}
 (u,u_t)\ \rm {is\ } H^1_0(D)\times L^2(D)\rm{-valued\ progressively\
 measurable},
\end{equation}
\begin{equation}\label{define1}
(u,u_t)\in L^2(\Omega; C([0,T];H_0^1(D)\times L^2(D))),\ \ \
u_t\in L^q((0,T)\times D),\ \ \
 \rm{for\ almost\ all }\ \omega,
\end{equation}
\begin{equation}\label{define2}
 u(0)=u_0,\ \ \ u_t(0)=u_1,
\end{equation}
\begin{equation}\label{define3}
u_{tt}-\Delta u+|u_t|^{q-2}u_t=|u|^{p-2}u+\varepsilon \sigma
(x,t)\partial_t W(t,x)
\end{equation}
holds in the sense of distributions over $(0,T)\times D$ for
almost all $\omega$.
\end{defn}
\begin{rem}
(\ref{define1}) and (\ref{define3}) imply that
\begin{eqnarray}\label{define5}
&&\big(u_t(t),\phi\big)=\big(u_1,\phi\big)-
\int_0^t\big(\nabla u,\nabla \phi\big)ds-\int_0^t\big(|u_s|^{q-2}u_s, \phi\big)ds\nonumber\\
  &&\quad\quad\quad\quad\quad\ +\int_0^t\big(|u|^{p-2}u, \phi\big)ds+\int_0^t\big(
  \phi,\varepsilon\sigma(x,s)dW_s\big),
\end{eqnarray}
for all $t\in[0,T]$ and all $\phi\in H_0^1(D)$. In fact,
(\ref{define5}) is a conventional form for the definition of
solution to stochastic differential equations. Here we say $u$ is
a strong solution of the equation (\ref{smain}).
\end{rem}

\section{ Existence and uniqueness of solution}

In this section, we deal with the local existence and uniqueness
of solution for problem (\ref{smain}) and  prove that the solution
of (\ref{smain}) is global for $q\geq p$. Let $f(u)=|u|^{p-2}u$.
For each $N\geq1$, define a $C^1$ function $\chi_N$ by
\begin{eqnarray*}
 \chi_N(x)=\left\{
                \begin{array}{ll}
                  1, \ \ \ & {\rm if}\ x\leq N,\\
                  \in (0,1), \ \ \ & {\rm if}\ N<x<N+1,\\
                  0 , \ \ \ & {\rm if}\ x\geq N+1,
                \end{array}
       \right.
\end{eqnarray*}
and further assume that $||\chi'_N||_\infty\leq2$. We define
\[
  f_N(u)= \chi_N(||\nabla u||_2)f(u),  \ \ \ u\in H_0^1(D).
\]
Then, it follows from (\ref{sobolev1}) that
\begin{equation}\label{condition6}
  ||f_N(u)-f_N(v)||_2\leq C_N||\nabla u-\nabla v||_2,  \ \ \ u,\ v\in
  H_0^1(D),
\end{equation}
where $C_N$ is a constant dependent only on $N$.  Let
$g(x)=|x|^{q-2}x$. For any $\lambda>0$, let
\[
 g_\lambda(x)=\frac{1}{\lambda}\big(x-(I+\lambda g)^{-1}(x)\big)=g(I+\lambda g)^{-1}(x),\ \ \ x\in \mathbb{R},
\]
where $g_\lambda$ is the Yosida approximation of the mapping $g$.
Since $g(x)$ satisfies maximal monotone and
$g'(x)=(q-2)|x|^{q-2}\geq0$ for any $x\in \mathbb{R}$, then
$g_\lambda\in C^1(\mathbb{R})$ and satisfies (see Pazy \cite{AP})
\begin{equation}\label{local}
  0\leq g'_\lambda\leq\frac{1}{\lambda},\ \ \ |g_\lambda (x)|\leq
  |g(x)|,\ \ \ |g_\lambda(x)|\leq \frac{1}{\lambda}|x|, \ {\rm\ for
  \ any
  }\ x\in\mathbb{R}.
\end{equation}

\begin{lem}\label{lem:regularize}
Let $\{\lambda_n\}$ be a sequence of positive numbers, and
$\{x_n\}$ be a sequence of real numbers such that
$\lambda_n\rightarrow0$ and $x_n\rightarrow x$. Then
\[
\lim_{n\rightarrow\infty}g_{\lambda_n}(x_n)=g(x).
\]
\end{lem}
\begin{proof}
There is some $L>0$ such that $|x_n|\leq L$ for all $n\geq 1$.
Since $g(x)$ is maximal monotone, let $y_n$ be a unique number
such that $y_n+\lambda_n g(y_n)=x_n$, for each $n\geq1$. Then we
have
\[
|y_n|\leq |x_n|\leq L,\ \ \ |x_n-y_n|\leq \lambda_n C,
\]
for each $n\geq1$, where $C=\sup_{|z|\leq L}|g(z)|$. Now the above
assertion follows from
\[
|g(x)-g_{\lambda_n}(x_n)|\leq|g(x)-g(x_n)|+|g(x_n)-g(y_n)|.
\]
\end{proof}

\begin{lem}\label{lem:regularize1}[See Lemma 1.3 in Lions \cite{LL}]
Let $D$ be a bounded domain in $\mathbb{R}^d,\ d\geq1$,
$\{\varphi_k\}$, $\varphi\in L^q(D),\ 1<q<\infty$. If
\[
||\varphi_k||_q\leq C\ \ \ {\rm and}\ \ \ \varphi_k(x)\rightarrow
\varphi(x)\ {\rm for\ almost\ all}\ x\in D,
\]
where $C$ is a constant, then $\varphi_k\rightarrow \varphi$
weakly in $L^q(D)$.
\end{lem}

In order to obtain the local existence and uniqueness of solution
for problem (\ref{smain}), we will first establish a lemma for the
regularized problem. Fix $\lambda$ and $N>0$, we will work on the
following initial boundary value problem
\begin{eqnarray}\label{regularize}
  \begin{cases}
     u_{tt}-\Delta u+g_\lambda(u_t)=f_N(u)+
     \varepsilon \sigma (x,t)\partial_t W(t,x),\ \ \ &(x,t)\in D \times (0,T),\\
    u(x,t)=0,            &(x,t)\in\partial D\times (0,T),\\
    u(x,0)=u_{0}(x),\ \ \ u_t(x,0)=u_1(x),\ \ \ &x\in D,
\end{cases}
\end{eqnarray}
where we suppose that
\begin{equation}\label{regularize1}
   (u_0,u_1)\in (H_0^1(D)\cap H^2(D))\times H_0^1(D)
\end{equation}
and that $\sigma(x,t)$ is $H_0^1(D)\cap L^\infty(D)$-valued
progressively measurable such that
\begin{equation}\label{regularize2}
\textbf{E}\int_0^T(||\nabla \sigma(t)||_2^2+||\sigma
(t)||_\infty^2)dt<\infty.
\end{equation}

\begin{lem}\label{lem:regularize2}
Assume (\ref{condition}), (\ref{regularize1}) and
(\ref{regularize2}) hold. Then there is a pathwise unique solution
$u$ of (\ref{regularize}) such that
\[
 u\in L^2\big(\Omega; L^\infty (0,T;H_0^1(D)\cap H^2(D))\big)\cap
 L^2\big(\Omega; C([0,T];H_0^1(D))\big)
\]
and
\[
 u_t\in L^2\big(\Omega; L^\infty (0,T;H_0^1(D))\big)\cap
 L^2\big(\Omega; C([0,T];L^2(D))\big).
\]
Moreover, it holds that
\[
 \textbf{E}\left(||u_t||^2_{L^\infty (0,T;H_0^1(D))}+||u||^2_{L^\infty (0,T;H_0^1(D)\cap H^2(D))}+\int_0^T\int_D
 g_\lambda(u_t)u_tdxdt\right)\leq C_N,
\]
where $C_N$ denotes a positive constant independent of $\lambda$.
\end{lem}
\begin{proof}
Let
\[
 u_m(t,x)=\sum_{j=1}^m a_{m,j}(t)e_j(x),
\]
where $\{e_j\}_{j=1}^\infty$ is a complete orthonormal base of
$H_0^1(D)$ satisfying (\ref{ef}) and $a_{m,j}$ form a solution of
the following system of stochastic differential equations
\begin{eqnarray}\label{regularize3}
\begin{cases}
a_{m,j}''=-\mu_ja_{m,j}-\Big(g_\lambda\big(\sum_{j=1}^m
a_{m,j}'e_j\big),e_j\Big)+\Big(f_N\big(\sum_{j=1}^m
a_{m,j}'e_j\big),e_j\Big) +\big(e_j,\varepsilon\sigma (x,t)dW_t\big),\\
a_{m,j}(0)=(u_0,e_j),\ \ \  a_{m,j}'(0)=(u_1,e_j),
\end{cases}
\end{eqnarray}
for $1\leq j\leq m$. By It\^{o} formula, we have
\begin{eqnarray}\label{regularize4}
&&||u'_m(t)||^2_2+||\nabla u_m(t)||^2_2\leq||u'_m(0)||^2_2+
||\nabla u_m(0)||^2_2-2\int_0^t\int_D g_\lambda \big(u_m'(s)\big)u_m'(s)dxds\nonumber\\
  && +2\int_0^t\int_D f_N \big(u_m\big)u_m'dx ds+2\int_0^t\big(
  u_m',\varepsilon\sigma\big)dW_s+c_0^2Tr R \sum_{j=1}^m\int_0^t |\big(e_j,\varepsilon\sigma\big )|^2ds,
\end{eqnarray}
and
\begin{eqnarray}\label{regularize5}
&&||\nabla u'_m(t)||^2_2+||\Delta u_m(t)||^2_2\leq ||\nabla
u'_m(0)||^2_2+||\Delta u_m(0)||^2_2+
2\int_0^t\int_D g_\lambda \big(u_m'(s)\big)\Delta u_m'(s)dxds\nonumber\\
  && -2\int_0^t\int_D f_N \big(u_m(s)\big)
  \Delta u_m'dx ds+2\int_0^t\big(
  \nabla u_m',\varepsilon\nabla (\sigma dW_s)\big)\nonumber\\
  && +2c_0^2Tr R \sum_{j=1}^m\int_0^t|\big(\nabla e_j,\varepsilon\nabla\sigma\big
  )|^2ds+2\sum_{j=1}^m\sum_{i=1}^\infty \lambda_i\int_0^t
  \big|(e_j, \sigma\nabla e_i)\big|^2ds
\end{eqnarray}
for all $t\in[0,T]$ and almost all $\omega$, where
\[
 Tr R=\sum_{i=1}^\infty \lambda_i, \ \ \
 c_0:=\sup_{i\geq1}||e_i||_\infty.
\]
From (\ref{sobolev}), (\ref{sobolev1}) and (\ref{local}), we get
\begin{equation}\label{regularize6}
  \int_D f_N \big(u_m\big)u_m'dx\leq \int_D\chi_N (||\nabla u_m||_2)|u_m|^{p-1}|u_m'(s)|dx\leq C_N||\nabla
  u_m||_2 ||u_m'||_2,
\end{equation}
\begin{eqnarray}\label{regularize7}
 &&-\int_D f_N \big(u_m\big)\Delta u_m'dx=
 (p-1)\int_D \chi_N(||\nabla u_m||_2)|u_m|^{p-2}\nabla u_m \cdot \nabla
 u_m'dx\nonumber\\
 &&\quad\quad\quad\quad\quad\quad\quad\quad\quad \ \leq C_N(p-1)|||u_m|^{p-2}\nabla u_m||_2||\nabla
 u_m'||_2\leq C_N||\Delta u_m||_2||\nabla
 u_m'||_2,\quad\quad
\end{eqnarray}
and
\begin{equation}\label{regularize8}
  \int_D g_\lambda \big(u_m'(s)\big)\Delta u_m'(s)dx=-\int_D g_\lambda' \big(u_m'(s)\big)|\nabla
  u_m'(s)|^2dx\leq0.
\end{equation}
By the Burkholder-Davis-Gundy inequality, we have
\begin{eqnarray}\label{regularize9}
 &&\textbf{E}\left(\sup_{t\in[0,T]}\left|\int_0^t\big(
  u_m'(s),\varepsilon\sigma\big)dW_s\right|\right)
  \leq C\textbf{E}\Big(\sup_{t\in[0,T]}||u_m'||_2\Big(\varepsilon^2\sum_{i=1}^\infty\int_0^T
  \big(\sigma(x,t)Re_i, \sigma(x,t)e_i\big)dt\Big)^{\frac{1}{2}}\Big)\nonumber\\
 &&
 \quad\quad\quad\quad\quad\quad\quad\quad\quad\quad\quad\quad\quad\quad\
\ \leq \alpha\textbf{E}\Big(\sup_{t\in[0,T]}||u_m'||_2^2\Big)
 +\frac{C\varepsilon^2c_0^2}{\alpha}Tr R\textbf{E}\Big(\int_0^T
  ||\sigma(t)||^2_2 dt\Big),\quad\quad
\end{eqnarray}
and
\begin{eqnarray}\label{regularize10}
 &&\textbf{E}\left(\sup_{t\in[0,T]}\left|\int_0^t\big(
  \nabla u_m',\nabla (\sigma dW_s)\big)\right|\right)
  \leq\alpha\textbf{E}\Big(\sup_{t\in[0,T]}||\nabla
  u_m'||_2^2\Big)\nonumber\\
  &&\quad\quad\quad\quad\quad\quad\quad\quad\quad\quad\quad\quad\quad\quad\quad\
  \
  +\frac{C\varepsilon^2c_0^2}{\alpha}Tr R\textbf{E}\Big(\int_0^T
  \big(||\nabla \sigma(t)
  ||_2^2+||\sigma(t)||_\infty^2\big)dt\Big).\quad\quad\quad
\end{eqnarray}
Here and below, $C$ and $C_N$ denote positive constants
independent of $m$ and $\lambda$. From (\ref{regularize1}),
(\ref{regularize2}) and (\ref{regularize4})--(\ref{regularize10}),
by Gronwall's inequality, we have
\begin{equation}\label{regularize11}
 \textbf{E}\left(\sup_{t\in[0,T]}||\nabla u_m'||_2^2+\sup_{t\in[0,T]}|| u_m||_{H^1_0(D)\cap H^2(D)}
 +\int_0^T\int_D g_\lambda
 \big(u_m'(s)\big)u_m'(s)dxds\right)\leq C_N.
\end{equation}

Define
\begin{equation}\label{regularize12}
 \mathcal{A}_\lambda=||v||^2_{L^\infty(0,T;H^1_0(D)\cap
 H^2(D))}+||v'||^2_{L^\infty(0,T;H^1_0(D))}+\int_0^T\int_D g_\lambda
 \big(v'(s)\big)v'(s)dxds.
\end{equation}
It follows from (\ref{regularize12}) that
\begin{equation}\label{regularize13}
P\Big(\bigcup_{L=1}^{\infty}\bigcap_{j=1}^{\infty}\bigcup_{m=j}^{\infty}\{\mathcal{A}_\lambda(u_m)\leq
L\}\Big)=1.
\end{equation}
Let $\mathcal{P}_m$ is the orthogonal projection of $L^2(D)$ onto
the subspace spanned by $\{e_1,\cdots, e_m\}$, i. e.,
\[
\mathcal{P}_m\varphi=\sum_{j=1}^m(\varphi,e_j)e_j.
\]
From (\ref{regularize3}), we have
\begin{equation}\label{regularize14}
\partial_t(u_m'-\varepsilon\mathcal{P}_m M(t))=\Delta u_m-\mathcal{P}_m
g_\lambda(u_m')+\mathcal{P}_mf_N(u_m)
\end{equation}
in the sense of distributions over $(0,T)\times D$ for almost all
$\omega$, where $M(t)$ is defined by (\ref{condition2}) with
(\ref{regularize2}). Since $\sigma(x,t)$ is $H_0^1(D)\cap
L^\infty(D)$-valued progressively measurable and
$\{W(t,x):t\geq0\}$ is a $V$-valued process, there is a subset
$\Omega_1\subset \Omega$ with $P(\Omega\setminus\Omega_1)=0$ such
that for each $\omega\in \Omega_1$,
\begin{equation}\label{regularize15}
M\in C([0,T];H_0^1(D)),\ \rm{and \ (\ref{regularize14})\ holds \
for \ all }\ m\geq1.
\end{equation}
From (\ref{regularize11}), for each $\omega\in \Omega_1$ there is
a subsequence $\{u_{m_k}\}_{k=1}^\infty$ such that
\begin{equation}\label{regularize16}
\mathcal{A}_\lambda(u_{m_k})\leq L_\omega,\ \rm{for\ all }\ k\geq1
\ \rm{and\ for \ some\ constant }\ L_\omega>0,
\end{equation}
\begin{equation}\label{regularize17}
u_{m_k}\rightarrow u\ \ \ \rm{weak \ star\ in}\ L^\infty(0,T;
H_0^1(D)\cap H^2(D)),
\end{equation}
\begin{equation}\label{regularize18}
u_{m_k}\rightarrow u\ \ \ \rm{strongly \ in  }\ C([0,T];
H_0^1(D)),
\end{equation}
and
\begin{equation}\label{regularize19}
u_{m_k}'\rightarrow u'\ \ \ \rm{weak \ star\ in}\ L^\infty(0,T;
H_0^1(D)),
\end{equation}
for some function $u=u(\omega)$. It follows from (\ref{local})
that
\[
 |g_\lambda(x)|^{\frac{q}{q-1}}\leq C g_\lambda(x)x,\ \ \ \rm{for \ all
 }\ x \in \mathbb{R} \ \rm{and} \ \lambda>0.
\]
From (\ref{condition}), we have the embedding
$L^{\frac{q}{q-1}}\subset H^{-1}(D)$. Thus, by
(\ref{regularize12}) and (\ref{regularize16}), we have
\begin{equation}\label{regulariz21}
 ||g_\lambda(u_{m_k}')||^{\frac{q}{q-1}}_{L^{\frac{q}{q-1}}(0,T;
 H^{-1}(D))}\leq C L_\omega,
\end{equation}
which combined with (\ref{regularize14}), yields
\begin{equation}\label{regularize20}
||u_{m_k}'-\varepsilon\mathcal{P}_{m_k}M||_{W^{1,\frac{q}{q-1}}(0,T;
 H^{-1}(D))}\leq C L_\omega
\end{equation}
for all $k\leq1$. By (\ref{regularize19}) and
(\ref{regularize20}), we get
\begin{equation}\label{regularize21}
u_{m_k}'-\varepsilon\mathcal{P}_{m_k}M\rightarrow u'-\varepsilon
M\ \ \ \rm{strongly\ in}\ C([0,T];L^2(D)).
\end{equation}
This implies that there is a subsequence still denoted by
$\{u_{m_k}\}$ such that
\begin{equation}\label{regulariz20}
u_{m_k}'(t,x)\rightarrow u'(t,x),\ \ \ \rm{for\ almost\ all}\
(t,x)\in (0,T)\times D.
\end{equation}
It follows from (\ref{regulariz21}), (\ref{regulariz20}) and Lemma
\ref{lem:regularize1} that
\[
g_\lambda(u_{m_k}')\rightarrow g_\lambda(u')\ \ \ \rm{weakly\ in}\
L^{\frac{q}{q-1}}((0,T)\times D).
\]
Thus, $u=u(\omega)$ satisfies (\ref{regularize}) in the sense of
distributions over $(0,T)\times D$. Here the choice of the above
subsequence may depend on $\omega\in \Omega_1$. If there is
another subsequence which converges to
$\widetilde{u}=\widetilde{u}(\omega)$ in the above sense, then
$w=u(\omega)-\widetilde{u}(\omega)$ satisfies
\[
w''-\Delta w
+g_\lambda(u'(\omega))-g_\lambda(\widetilde{u}'(\omega))=f_N(u(\omega))-f_N(\widetilde{u}(\omega)),
\]
\[
w(0)=0,\ \ \ w'(0)=0,
\]
\[
w\in L^\infty (0,T;H_0^1(D)\cap H^2(D))\cap C([0,T);H_0^1(D)),
\]
\[
w'\in L^\infty (0,T;H_0^1(D))\cap C([0,T);L^2(D)).
\]
Thus, we have
\begin{equation}\label{regularize22}
\frac{1}{2}\frac{d}{dt}(||w'(t)||^2_2+||\nabla w(t)||_2^2)+\int_D
\big(g_\lambda(u')-g_\lambda(\widetilde{u}')\big)w'dx=\int_D\big(f_N(u)-f_N(\widetilde{u})\big)w'dx.
\end{equation}
From (\ref{local}), we get
\[
\int_D
g_\lambda(u'(\omega))-g_\lambda(\widetilde{u}'(\omega))w'dx\geq0.
\]
By the H\"{o}lder inequality, it follows from (\ref{condition})
that
\begin{eqnarray}\label{regularize23}
  &&\left|\int_D(f_N(u)-f_N(\widetilde{u})\big)w'dx\right|=
  \left|\int_D\big(\chi_N(||\nabla u||_2)|u|^{p-2}u-
  \chi_N(||\nabla \widetilde{u}||_2)|\widetilde{u}|^{p-2}\widetilde{u}\big)w'dx\right|\nonumber\\
  &&\leq C_N(p-1)\int_D \sup \{|u|^{p-2},|\widetilde{u}|^{p-2}\}|w||w'|dx
  \leq C_N(||u||_{(p-2)d}^{(p-2)}+||\widetilde{u}||_{(p-2)d}^{(p-2)})||w||_{\frac{2d}{d-2}}
  ||w'||_2\nonumber\\
  &&\leq C_N ||\nabla w(t)||_2 ||w'||_2.
\end{eqnarray}
Combining with (\ref{regularize22}) with (\ref{regularize23}), we
have
\[
||w'(t)||^2_2+||\nabla w(t)||_2^2\leq 2C_N\int_0^t
(||w'(s)||^2_2+||\nabla w(s)||_2^2 )ds,
\]
which implies $w=0$, i.e., $u(\omega)=\widetilde{u}(\omega)$.
Hence, for each $\omega\in \Omega_1$, $u=u(\omega)$ is
well-defined.

We shall also show that $(u,u_t)$ is $(H_0^1(D)\cap H^2(D))\times
H^1_0(D)$-valued progressively measurable for any $0\leq t\leq T$.
Let $\textbf{B}_r(z)$ be a closed ball in $C([0,T];H_0^1(D)\times
L^2(D))$ with radius $r>0$ and center at $z$. Then by virtue of
the way $u$ has been obtained, it holds that
\begin{eqnarray}\label{regularize24}
 \{(u,u_t)\in
 \textbf{B}_r(z)\}\cap\Omega_1=\Omega_1\cap
 \bigcup_{L=1}^\infty\bigcap_{\nu=1}^\infty\bigcap_{j=1}^\infty\bigcup_{m=j}^\infty
 \left\{\big((u_m,u_m')\in \textbf{B}_{r+1/\nu}(z)\big)\cap\big(\mathcal{A}_\lambda(u_{m})\leq
 L\big)\right\}.
\end{eqnarray}
Since $(u,u_t)\in C([0,T];H_0^1(D)\times L^2(D))$ for almost all
$\omega$, and the right- hand side of (\ref{regularize24}) belongs
to $\mathcal{F}_T$, it holds that
\begin{eqnarray}\label{regularize25}
 \big\{(t,\omega)|0\leq t\leq T, (u(t,\omega),u_t(t,\omega))\in
 A\big\}\in \mathcal{B}([0,T])\otimes \mathcal{F}_T,
\end{eqnarray}
for every $A\in \mathcal{B}(H_0^1(D)\times L^2(D))$. Since every
closed ball of finite radius in $(H_0^1(D)\cap H^2(D))\times
H^1_0(D)$ is closed in $ H^1_0(D)\times L^2(D)$, we have
$\mathcal{B}((H_0^1(D)\cap H^2(D))\times H^1_0(D))\subset
\mathcal{B}(H^1_0(D)\times L^2(D))$. Thus, (\ref{regularize25})
holds for every $\mathcal{B}((H_0^1(D)\cap H^2(D))\times
H^1_0(D))$. By the pathwise uniqueness, we may replace $T$ in
(\ref{regularize25}) by any $0\leq t\leq T$ and $(u,u_t)$ is
$(H_0^1(D)\cap H^2(D))\times H^1_0(D)$-valued progressively
measurable.

Next we show that for each $\omega\in\Omega_1$,
\begin{eqnarray}\label{regularize26}
\mathcal{A}_\lambda(u)\wedge K\leq
\underline{\lim}_{m\rightarrow\infty}\mathcal{A}_\lambda(u_m)\wedge
K
\end{eqnarray}
for each $K>0$. If
$\underline{\lim}_{m\rightarrow\infty}\mathcal{A}_\lambda(u_m)\wedge
K=K$, then the inequality is obvious. If
$\underline{\lim}_{m\rightarrow\infty}\mathcal{A}_\lambda(u_m)\wedge
K=\delta<K$, then there is a subsequence
$\{u_{m_k}\}_{k=1}^\infty$ such that
\[
\lim_{k\rightarrow\infty}\mathcal{A}_\lambda(u_{m_k})=\delta,
\]
and $\{u_{m_k}(\omega)\}$ converges to $u(\omega)$ in the sense of
(\ref{regularize16})-(\ref{regularize19}) and
(\ref{regularize21}).  It follows that
\[
 ||u||_{L^\infty(0,T;H_0^1(D)\cap H^2(D))}\leq
 \underline{\lim}_{k\rightarrow\infty}||u_{m_k}||_{L^\infty(0,T;H_0^1(D)\cap
 H^2(D))},
\]
\[
 ||u'||_{L^\infty(0,T;H_0^1(D))}\leq
 \underline{\lim}_{k\rightarrow\infty}||u_{m_k}'||_{L^\infty(0,T;H_0^1(D))},
\]
and
\[
\int_0^T\int_D g_\lambda
 \big(u'(s)\big)u'(s)dxds\leq \underline{\lim}_{k\rightarrow\infty}\int_0^T\int_D g_\lambda
 \big(u_{m_k}'(s)\big)u_{m_k}'(s)dxds,
\]
which yield
\[
\mathcal{A}_\lambda(u)\leq\delta.
\]
Thus, (\ref{regularize26}) is valid. By (\ref{regularize11}),
(\ref{regularize26}) and Fatou's lemma, we have
\[
\textbf{E}(\mathcal{A}_\lambda(u)\wedge K)\leq C_N,
\]
for some constant $C_N$ independent of $K$ and $\lambda$. By
passing $K\uparrow \infty$, we get
\begin{equation}\label{vary0}
\textbf{E}(\mathcal{A}_\lambda(u))\leq C_N.
\end{equation}
\end{proof}

Next we still fix $N>0$ and consider the following equation
\begin{eqnarray}\label{vary}
  \begin{cases}
     u_{tt}-\Delta u+g(u_t)=f_N(u)+\varepsilon \sigma (x,t)\partial_t W(t,x),\ \ \ &(x,t)\in D \times (0,T),\\
    u(x,t)=0,            &(x,t)\in\partial D\times (0,T),\\
    u(x,0)=u_{0}(x),\ \ \ u_t(x,0)=u_1(x),\ \ \ &x\in D.
\end{cases}
\end{eqnarray}

\begin{lem}\label{lem:vary}
Assume (\ref{condition}), (\ref{regularize1}) and
(\ref{regularize2}) hold. Then there is a pathwise unique solution
$u$ of (\ref{vary}) such that
\[
 u\in L^2\big(\Omega; L^\infty (0,T;H_0^1(D)\cap H^2(D))\big)\cap
 L^2\big(\Omega; C([0,T];H_0^1(D))\big),
\]
\[
 u_t\in L^2\big(\Omega; L^\infty (0,T;H_0^1(D))\big)\cap
 L^2\big(\Omega; C([0,T];L^2(D))\big),
\]
and
\[
u_t\in L^q((0,T)\times D).
\]
\end{lem}
\begin{proof}
We denote by $u_\lambda$ the solution of (\ref{regularize}) under
the conditions (\ref{regularize1}) and (\ref{regularize2}). Since
$\textbf{E}(\mathcal{A}_\lambda(u_\lambda))\leq C_N$ for all
$\lambda>0$, we can repeat the same argument as above by
considering $\lambda=\frac{1}{m}$, $m=1,2,\cdots$. there is
$\Omega_2\subset \Omega$ with $P(\Omega\setminus\Omega_2)=0$ and
the following properties. For each $\omega\in\Omega_2$,
\begin{equation}\label{vary1}
M\in C([0,T];H_0^1(D)),\ \rm{and  \ for \ all }\
\lambda=\frac{1}{m},\ m\geq1,
\end{equation}
\[
( u_\lambda'-\varepsilon M(t))'-\Delta u_\lambda+
g_\lambda(u_\lambda')=f_N(u_\lambda)
\]
holds in the sense of distributions over $(0,T)\times D$, and
there is a subsequence satisfying the following.
\begin{equation}\label{vary2}
\mathcal{A}_{\lambda_k}(u_{\lambda_k})\leq L_\omega,\ \rm{for\ all
}\ k\geq1 \ \rm{and\ for \ some\ constant }\ L_\omega>0,
\end{equation}
\begin{equation}\label{vary3}
u_{\lambda_k}\rightarrow u\ \ \ \rm{weak \ star\ in}\
L^\infty(0,T; H_0^1(D)\cap H^2(D)),
\end{equation}
\begin{equation}\label{vary4}
u_{\lambda_k}\rightarrow u\ \ \ \rm{strongly \ in  }\ C([0,T];
H_0^1(D)),
\end{equation}
\begin{equation}\label{vary5}
 u_{\lambda_k}'\rightarrow u'\ \ \ \rm{weak \ star\ in}\
L^\infty(0,T; H_0^1(D)),
\end{equation}
\begin{equation}\label{vary5}
u_{\lambda_k}'\rightarrow u'\ \ \ \rm{strongly \ in  }\ C([0,T];
L^2(D)),
\end{equation}
and
\begin{equation}\label{vary5}
 u_{\lambda_k}'\rightarrow u'\ \ \ \rm{for \ almost\ all  }\ (x,t)\in (0,T)\times
D,
\end{equation}
for some function $u=u(\omega)$. By Lemma \ref{lem:regularize},
\[
g_{\lambda_k}(u_{\lambda_k}')\rightarrow g( u')\ \ \ \rm{for \
almost\ all  }\ (x,t)\in (0,T)\times D.
\]
It follows from (\ref{vary2}) and Lemma \ref{lem:regularize1} that
\[
g_{\lambda_k}(u_{\lambda_k}')\rightarrow g(u')\ \ \ \rm{weakly\
in}\ L^{\frac{q}{q-1}}((0,T)\times D).
\]
Thus, $u=u(\omega)$ satisfies (\ref{vary}) in the sense of
distributions over $(0,T)\times D$ for $\omega\in\Omega$.  Suppose
that for $\omega\in \Omega$, there is another subsequence which
converges to $\widetilde{u}=\widetilde{u}(\omega)$ in the sense of
(\ref{vary2})-(\ref{vary5}). Similarly the proof in Lemma
\ref{lem:regularize2}, we can show that
$u(\omega)=\widetilde{u}(\omega)$ follows from the equation
\[
u_{tt}(\omega)-\widetilde{u}_{tt}(\omega)-\Delta
(u(\omega)-\widetilde{u}(\omega))
+g(u_t(\omega))-g(\widetilde{u}_t(\omega))=f_N(u(\omega))-f_N(\widetilde{u}(\omega)),
\]
and the regularity
\[
u(\omega), \widetilde{u}(\omega)\in L^\infty (0,T;H_0^1(D)\cap
H^2(D))\cap C([0,T);H_0^1(D)),
\]
\[
u_t(\omega), \widetilde{u}_t(\omega)\in L^\infty
(0,T;H_0^1(D))\cap C([0,T);L^2(D)),
\]
\[
g(u_t(\omega)), g(\widetilde{u}_t(\omega))\in
L^{\frac{q}{q-1}}((0,T)\times D).
\]

Again by the same argument as Lemma \ref{lem:regularize2},
$(u,u_t)$ is $(H_0^1(D)\cap H^2(D))\times H_0^1(D)$-valued
progressively measurable. Next we define
\begin{equation}\label{vary6}
 \mathcal{A}(u)=||u||^2_{L^\infty(0,T;H^1_0(D)\cap
 H^2(D))}+|| u_t||^2_{L^\infty(0,T;H^1_0(D))}+\int_0^T\int_D g_\lambda
 (u_t)u_tdxdt.
\end{equation}
Then by the same argument as (\ref{vary0}), we have
\begin{equation}\label{vary7}
\textbf{E}\big(\mathcal{A}(u)\big)\leq C_N.
\end{equation}
\end{proof}

Now we consider the local existence and uniqueness of solution for
problem (\ref{smain}) under the assumption (\ref{mild}).
\begin{thm}\label{thm local}
Under the assumptions (\ref{condition}), (\ref{mild}) and
(\ref{mild1}), there is a pathwise unique local solution $u$ of
(\ref{smain}) according to Definition \ref{defn} such that the
energy equation holds:
\begin{eqnarray}\label{ienery}
&&||\nabla u(t)||_2^2+||u_t(t)||^2_2+2\int_0^t\int_D |u_t(s)|^qdxds-2\int_0^t\int_D |u(s)|^{p-2}u(s)u_t(s)dxds\nonumber\\
  &&=||\nabla u_0||^2_2+||u_1||^2_2+2\int_0^t\big(u_t(s),\varepsilon\sigma(x,s)\big)dW_s+\varepsilon^2\sum_{i=1}^\infty\int_0^t
  \int_D \lambda_ie_i^2(x)\sigma^2(x,s)dxds.
\end{eqnarray}
\end{thm}
\begin{proof}
Let us choose sequences $\{u_{0,m}\}$, $\{u_{1,m}\}$ and
$\{\sigma_m(x,t,\omega)\}$ such that
\[
 u_{0,m}\in H^1_0(D)\cap H^2(D),\ \ \ u_{1,m}\in H^1_0(D),\ \ \ \sigma_m(x,t,\omega)\in L^2(\Omega;L^2(0,T;H^1_0(D)\cap L^\infty(D)))
\]
\[
\textbf{E}\int_0^T(||\nabla \sigma_m(t)||_2^2+||\sigma_m
(t)||_\infty^2)dt<\infty.,
\]
and as $m\rightarrow\infty$,
\begin{equation}\label{exist1}
u_{0,m}\rightarrow u_0\ \ \ \rm{strongly\ in}  \ \ \ H^1_0(D),
\end{equation}
\begin{equation}\label{exist2}
u_{1,m}\rightarrow u_1\ \ \ \rm{strongly\ in}  \ \ \ L^2(D),
\end{equation}
\begin{equation}\label{exist3}
\textbf{E}\int_0^T ||\sigma_m(x,t)-\sigma(x,t)||_2^2 dt\rightarrow
0.
\end{equation}
For each $m\geq1$, let $u_m$ be the solution of
\begin{eqnarray}\label{exist4}
  \begin{cases}
     u_{tt}-\Delta u+g(u_t)=f_N(u)+\varepsilon \sigma_m (x,t)\partial_t W(t,x),\ \ \ &(x,t)\in D \times (0,T),\\
    u(x,t)=0,            &(x,t)\in\partial D\times (0,T),\\
    u(x,0)=u_{0,m}(x),\ \ \ u_t(x,0)=u_{1,m}(x),\ \ \ &x\in D.
\end{cases}
\end{eqnarray}
By Lemma \ref{lem:vary}, we have
\begin{equation}\label{exist5}
 u_m\in L^2\big(\Omega; L^\infty (0,T;H_0^1(D)\cap H^2(D))\big)\cap
 L^2\big(\Omega; C([0,T];H_0^1(D))\big),
\end{equation}
\begin{equation}\label{exist6}
 u_m'\in L^2\big(\Omega; L^\infty (0,T;H_0^1(D))\big)\cap
 L^2\big(\Omega; C([0,T];L^2(D))\big),
\end{equation}
and the energy equation
\begin{eqnarray}\label{ienery1}
&&||\nabla u_m||_2^2+||u_m'||^2_2+2\int_0^t\int_D |u_m'|^qdxds
-2\int_0^t\int_D \chi(||\nabla u_m||_2)|u_m|^{p-2}u_mu_m'(s)dxds\nonumber\\
  &&=||\nabla u_{0,m}||^2_2+||u_{1,m}||^2_2+2\int_0^t\big(u_m',\varepsilon\sigma_m\big)dW_s+\varepsilon^2\sum_{i=1}^\infty\int_0^t
  \int_D \lambda_ie_i^2(x)\sigma_m^2(x,s)dxds.
\end{eqnarray}
Let
\[
M_m(t,x)=\int_0^t \sigma_m(x,s)dW(s,x),\ \ \ t>0,\ x\in D.
\]
Then, for any $m_1$, $m_2$
\begin{equation}\label{exist7}
(u_{m_1}''-u_{m_2}'')-\Delta
(u_{m_1}-u_{m_2})+g(u_{m_1}')-g(u_{m_2}')=f_N(u_{m_1})-f_N(u_{m_2})+\varepsilon(M_{m_1}-M_{m_2})'
\end{equation}
holds in the sense of distributions over $(0,T)\times D$ for
almost all $\omega$. For the damping term, we use the following
elementary inequality
\begin{equation}\label{iexist}
 (|a|^{q-2}a-|b|^{q-2}b)(a-b)\geq c|a-b|^{q}
\end{equation}
for $a,\ b\in \mathbb{R}$, $q\geq2$, where $c$ is a positive
constant. By inequality (\ref{iexist}) and  the regularity
(\ref{exist5}) and (\ref{exist6}), we can drive from
\begin{eqnarray}\label{exist8}
 &&||u_{m_1}'(t)-u_{m_2}'(t)||_2^2+||\nabla u_{m_1}(t)-\nabla u_{m_2}(t)||^2_2+2c\int_0^t
 ||u_{m_1}'-u_{m_2}'||^q_qds\nonumber\\
 &&\leq ||\nabla u_{0,m_1}-\nabla u_{0,m_2}||^2_2+|| u_{1,m_1}-u_{1,m_2}||^2_2+2\int_0^t\big(f_N(u_{m_1})-f_N(u_{m_2}),
 u_{m_1}'-u_{m_2}'\big)ds\nonumber\\
&&\ \ \ \ +2\varepsilon\int_0^t\big(\sigma_{m_1}-\sigma_{m_2},
 u_{m_1}'-u_{m_2}'\big)dW_s+\varepsilon^2 c_0^2Tr R\int_0^t||\sigma_{m_1}-\sigma_{m_2}||_2^2ds
\end{eqnarray}
for all $t\in [0,T]$. For the third term on the right of
(\ref{exist8}), it follows from (\ref{condition6}) that
\begin{eqnarray}\label{exist9}
 &&2\Big|\int_0^t\big(f_N(u_{m_1})-f_N(u_{m_2}),
 u_{m_1}'-u_{m_2}'\big)ds\Big|\leq 2\int_0^t ||f_N(u_{m_1})-f_N(u_{m_2})||_2||u_{m_1}'-u_{m_2}'||_2ds\nonumber\\
 &&\leq 2C_N\int_0^t||\nabla u_{m_1}-\nabla u_{m_2}||_2||u_{m_1}'-u_{m_2}'||_2ds\nonumber\\
 &&\leq C_N\int_0^t||\nabla u_{m_1}-\nabla
 u_{m_2}||_2^2dt+C_N\int_0^t||u_{m_1}'-u_{m_2}'||_2^2ds,
\end{eqnarray}
where $C_N$ ia s positive constant independent of $m_1$ and $m_2$.
In view of (\ref{exist8}) and (\ref{exist9}), it follows that
\begin{eqnarray}\label{exist13}
 &&\textbf{E}\sup_{0\leq t\leq T}\Big(||u_{m_1}'(t)-u_{m_2}'(t)||_2^2+
 ||\nabla u_{m_1}(t)-\nabla u_{m_2}(t)||^2_2\Big) \nonumber\\
&&\leq ||\nabla u_{0,m_1}-\nabla
u_{0,m_2}||^2_2+C_N\int_0^T\textbf{E}\sup_{0\leq t\leq
T}\Big(||\nabla u_{m_1}-\nabla
 u_{m_2}||_2^2+||u_{m_1}'-u_{m_2}'||_2^2\Big)ds\nonumber\\
&&\ \ \ + ||
u_{1,m_1}-u_{1,m_2}||^2_2+\varepsilon^2 c_0^2Tr R\textbf{E}\int_0^t||\sigma_{m_1}-\sigma_{m_2}||_2^2ds\nonumber\\
&&\ \ \ +2\varepsilon\textbf{E}\sup_{0\leq t\leq
T}\Big|\int_0^t\big(\sigma_{m_1}-\sigma_{m_2},
 u_{m_1}'-u_{m_2}'\big)dW_s\Big| .
\end{eqnarray}
For the last term on the right of (\ref{exist13}), by the
Burkholder-Davis-Gundy inequality we have
\begin{eqnarray}\label{exist14}
 &&\textbf{E}\Big(\sup_{t\in[0,T]}\Big|\int_0^t\big(\sigma_{m_1}-\sigma_{m_2},
 u_{m_1}'-u_{m_2}'\big)dW_s\Big|\Big)\nonumber\\
 &&\leq C\textbf{E}\Big(\sup_{t\in[0,T]}||u_{m_1}'-u_{m_2}'||_2\Big(\sum_{i=1}^\infty\int_0^t
 \big((\sigma_{m_1}-\sigma_{m_2})Re_i,(\sigma_{m_1}-\sigma_{m_2})e_i\big)dt
 \Big)^{\frac{1}{2}}\Big)\nonumber\\
 &&\leq
 \alpha\textbf{E}\Big(\sup_{t\in[0,T]}||u_{m_1}'-u_{m_2}'||_2^2\Big)
 +\frac{Cc_0^2}{\alpha}Tr R\textbf{E}\int_0^t||\sigma_{m_1}-\sigma_{m_2}||^2_2dt
\end{eqnarray}
where $\alpha$ and $C$ are some positive constants. By
taking(\ref{exist13}), (\ref{exist14}) into account and invoking
the Gronwall inequality again, we get
\begin{eqnarray}\label{exist13'}
  &&\textbf{E}\Big(\sup_{t\in[0,T]}||u_{m_1}'-u_{m_2}'||_2^2+\sup_{t\in[0,T]}||\nabla u_{m_1}-\nabla
 u_{m_2}||_2^2\Big)\nonumber\\
  &&\leq C_N \Big(||\nabla u_{0,m_1}-\nabla u_{0,m_2}||^2_2+||
  u_{1,m_1}-u_{1,m_2}||^2_2+\varepsilon^2 c_0^2Tr
  R\textbf{E}\int_0^t||\sigma_{m_1}-\sigma_{m_2}||_2^2ds\Big).
\end{eqnarray}
Moreover, it can be derived from (\ref{exist8}) and
(\ref{exist13'}) that
\begin{eqnarray}\label{exist12}
 &&\textbf{E}\Big(\int_0^T
 \sup_{t\in[0,T]}||u_{m_1}'-u_{m_2}'||^q_qdt\Big)\nonumber\\
 &&\leq C_N \Big(||\nabla u_{m_1}(t)-\nabla
 u_{m_2}(t)||^2_2
 +||
u_{1,m_1}-u_{1,m_2}||^2_2 +\varepsilon^2 c_0^2Tr
  R\textbf{E}\int_0^t||\sigma_{m_1}-\sigma_{m_2}||_2^2ds\Big).\quad
\end{eqnarray}
It follows from (\ref{exist1})-(\ref{exist3}) and (\ref{exist13'})
that $\{u_m\}$ and $\{u_m'\}$ are cauchy sequences in
$L^2(\Omega;H_0^1(D))$ and $L^2(\Omega;L^2(D))$, respectively.
Thus,
\begin{equation}\label{exist11}
(u_m, u_m')\rightarrow (u_N,u_N') \ \ \ \rm{stronly \ in }\ \ \
L^2(\Omega;C([0,T];H_0^1(D)\times L^2(D)))
\end{equation}
for some function $u_N$ dependent on $N$. Also, by
(\ref{exist12}), $\{u_m'\}$ are cauchy sequences in
$L^q((0,T)\times D)$. So $\{u_m'\}$ converge strongly in
$L^q((0,T)\times D)$. Then there exist subsequences of $\{u_m'\}$,
still denoted by $\{u_m'\}$ such that
\begin{equation}\label{exist15}
u_m'\rightarrow u_N'\ \ \ \rm {for \ almost\ all}\ (x,t)\in
(0,T)\times D.
\end{equation}
It follows from (\ref{vary7}), (\ref{exist15}) and Lemma
\ref{lem:regularize1} that
\begin{equation}\label{exist16}
|u_m'|^{q-2}u_m'\rightarrow |u_N'|^{q-2}u_N'\ \ \ \rm{weakly\ in}\
L^{\frac{q}{q-1}}((0,T)\times D).
\end{equation}
Therefore, using (\ref{exist11}), (\ref{exist15}), the convergence
of the initial data and $\sigma_m (x,t)$, $u_N$ is the solution of
the following equation
\begin{eqnarray}\label{exist17}
  \begin{cases}
     u_{tt}-\Delta u+|u_t|^{q-2}u_t=f_N(u)+\varepsilon \sigma (x,t)\partial_t W(t,x),\ \ \ &(x,t)\in D \times (0,T),\\
    u(x,t)=0,            &(x,t)\in\partial D\times (0,T),\\
    u(x,0)=u_{0}(x),\ \ \ u_t(x,0)=u_{1}(x),\ \ \ &x\in D,
\end{cases}
\end{eqnarray}
which satisfies the requirements of Definition \ref{defn}, where
$u_0$, $u_1$ and $\sigma (x,t)$ satisfy condition (\ref{mild}).
For uniqueness of (\ref{exist17}), the proof is similar in the
Lemma \ref{lem:regularize2}, so we omit it here.

To obtain the energy equation of (\ref{exist17}), we proceed by
taking the termwise limit in the approximate equation
(\ref{ienery1}). It is ease to show that
\[
 ||\nabla u_m||_2^2\rightarrow ||\nabla u_N||_2^2,\ \ \  || u_m'||_2^2\rightarrow ||
 u_N'||_2^2, \ \ \ ||\nabla u_{0,m}||_2^2\rightarrow ||\nabla
 u_0||_2^2, \ \ \ || u_{1,m}||_2^2\rightarrow ||
 u_1||_2^2
\]
in the mean and
\[
\int_0^t\int_D |u_m'|^{q}\rightarrow \int_0^t\int_D |u_N'|^{q}
\]
by (\ref{exist16}). By the dominated convergence theorem, the term
$\int_0^t
  \int_D \lambda_1e_i^2(x)\sigma_m^2(x,s)dxds$ converges in the mean to
$\int_0^t
  \int_D \lambda_1e_i^2(x)\sigma^2(x,s)dxds$. For the remaining two terms in
(\ref{ienery1}), we first consider
\begin{eqnarray}\label{exist18}
&&\Big|\int_0^t\int_D \chi(||\nabla
u_m||_2)|u_m|^{p-2}u_mu_m'(s)dxds-\int_0^t\int_D \chi(||\nabla
u_N||_2)|u_N|^{p-2}u_Nu_N'(s)dxds\Big|\nonumber\\
&&\leq\int_0^t\Big|\big(f_N(u_m)-f_N(u_N),u_N'\big)\Big|ds
+\int_0^t\Big|\big(f_N(u_m),u_m'-u_N'\big)\Big|ds.
\end{eqnarray}
From (\ref{condition6}) and (\ref{sobolev}), we get
\begin{equation}\label{exist19}
  \Big|\big(f_N(u_m)-f_N(u_N),u_N'\big)\Big|\leq
  ||f_N(u_m)-f_N(u_N)||_2||u_N'||_2\leq C_N ||\nabla u_m-\nabla
  u_N||_2||u_N'||_2,
\end{equation}
and
\begin{equation}\label{exist20}
  \Big|\big(f_N(u_m),u_m'-u_N'\big)\Big|\leq C_N ||\nabla
  u_m||_2||u_m'-u_N'||_2.
\end{equation}
substituting (\ref{exist19}) and (\ref{exist20}) into
(\ref{exist18}), we obtain
\begin{eqnarray*}
&&\Big|\int_0^t\int_D \chi(||\nabla
u_m||_2)|u_m|^{p-2}u_mu_m'(s)dxds-\int_0^t\int_D \chi(||\nabla
u_N||_2)|u_N|^{p-2}u_Nu_N'(s)dxds\Big|\nonumber\\
&&\leq C_N\int_0^t(||\nabla
  u_m||_2+||u_N'||_2)(||\nabla u_m-\nabla u_N||_2+||u_m'-u_N'||_2)ds.
\end{eqnarray*}
Therefore
\begin{eqnarray*}
&&\textbf{E}\Big|\int_0^t\int_D \chi(||\nabla
u_m||_2)|u_m|^{p-2}u_mu_m'(s)dxds-\int_0^t\int_D \chi(||\nabla
u_N||_2)|u_N|^{p-2}u_Nu_N'(s)dxds\Big|^2\nonumber\\
&&\leq 2C_N\Big(\textbf{E}\int_0^T(||\nabla
  u_m||_2^2+||u_N'||_2^2)ds\Big)\Big(\textbf{E}\int_0^T(||\nabla u_m-\nabla
  u_N||_2+||u_m'-u_N'||_2)ds\Big),
\end{eqnarray*}
which converges to zero as $m\rightarrow\infty$. Finally, for the
stochastic integral term, we have
\begin{eqnarray*}
&&\textbf{E}\Big|\int_0^t\big(u_m',\sigma_m\big)dW_s-\int_0^t\big(u_N',\sigma\big)dW_s\big|\\
&&\leq
\textbf{E}\Big|\int_0^t\big(u_m'-u_N',\sigma_m\big)dW_s\big|
+\textbf{E}\Big|\int_0^t\big(u_N',\sigma_m-\sigma\big)dW_s\big|.
\end{eqnarray*}
Now, by the Burkholder-Davis-Gundy inequality, we have
\begin{eqnarray*}
&& \textbf{E}\sup_{0\leq t\leq
T}\Big|\int_0^t\big(u_m'-u_N',\sigma_m\big)dW_s\big|\leq
C\textbf{E}\Big(\sup_{0\leq t\leq
T}||u_m'-u_N'||_2\Big(\sum_{i=1}^\infty\int^T_0(\sigma_{m}R e_i,\sigma_{m}e_i)dt\Big)^{\frac{1}{2}}\Big)\quad\\
&&\leq Cc_0Tr R\textbf{E}\Big(\sup_{0\leq t\leq
T}||u_m'-u_N'||_2^2\Big)^{\frac{1}{2}}\textbf{E}
\Big(\int^T_0||\sigma_{m}||^2_2dt\Big)^{\frac{1}{2}}\rightarrow0,\
\ \ \rm{as}\ m\rightarrow\infty.
\end{eqnarray*}
Similarly,
\[
 \textbf{E}\sup_{0\leq t\leq
T}\Big|\int_0^t\big(u_N',\sigma_m-\sigma\big)dW_s\big| \leq Cc_0Tr
R\textbf{E}\Big(\sup_{0\leq t\leq
T}||u_N'||_2^2\Big)^{\frac{1}{2}}\textbf{E}
\Big(\int^T_0||\sigma_m-\sigma||_2^2dt\Big)^{\frac{1}{2}},
\]
which also tends to zero as $m\rightarrow0$ by (\ref{exist3}).
There above three inequalities imply that
\[
\int_0^t\big(u_m',\sigma_m\big)dW_s\rightarrow\int_0^t\big(u_N',\sigma\big)dW_s,\
\ \ \rm{as}\ m\rightarrow\infty.
\]
Hence, we obtain the energy equation of (\ref{exist17})
\begin{eqnarray}\label{ienery2}
&&||\nabla u_N||_2^2+||u_N'||^2_2+2\int_0^t\int_D |u_N'|^qdxds
-2\int_0^t\int_D \chi(||\nabla u_N||_2)|u_N|^{p-2}u_Nu_N'(s)dxds\nonumber\\
  &&=||\nabla u_{0}||^2_2+||u_{1}||^2_2+2\int_0^t\big(u_N',\varepsilon\sigma\big)dW_s+\varepsilon^2\sum_{i=1}^\infty\int_0^t
  \int_D \lambda_1e_i^2(x)\sigma^2(x,s)dxds.
\end{eqnarray}

For each $N$, introduce the stopping time $\tau_N$ by
\[
\tau_N=\inf\{t>0;||\nabla u_N||_2\geq N\}.
\]
By the uniqueness of the solution of (\ref{exist17}), for
$t\in[0,\tau_N\wedge T)$, $u(t)=u_N(t)$ is the local solution of
(\ref{smain}). As $\tau_N$ is increasing in $N$, let
$\tau_\infty=\lim_{N\rightarrow\infty}\tau_N$. Hence, we construct
a unique continuous local solution $u(t)=\lim_{N\rightarrow\infty}
u_N(t)$ to (\ref{smain}) on $[0,T\wedge\tau_\infty)$, which
satisfies the requirements of Definition \ref{defn} and the energy
equation (\ref{ienery}).
\end{proof}

To obtain a global solution, it is necessary to consider the
interaction between the damping term $|u_t|^{q-2}u_t$ and the
source term $|u|^{p-2}u$ such that a certain energy bound can be
established to prevent the unlimited growth. To state the next
theorem, we define
\[
e\big(u(t)\big)=||u_t(t)||_2^2+||\nabla
u(t)||_2^2+\frac{2}{p}||u||_p^p.
\]
\begin{thm}\label{thm global}
Suppose (\ref{condition}),  (\ref{mild}) and (\ref{mild1}) hold.
If $q\geq p$, then for any $T>0$, there is a unique solution $u$
of (\ref{smain}) according to Definition \ref{defn} on the
interval $[0,T]$ such that
\begin{equation}\label{global0}
 \textbf{E}\sup_{0\leq t\leq T}e(t)<\infty.
\end{equation}
\end{thm}
\begin{proof}
For any $T>0$, we will show that
$u_N(t)=u(t\wedge\tau_N)\rightarrow u$ a.s. as
$N\rightarrow\infty$ for any $t\leq T$, so that the local solution
becomes a global one. To this end, it suffices to show that
$\tau_N\rightarrow\infty$ as $N\rightarrow\infty$ with probability
one.

Recall that, for $t\in[0,\tau_N\wedge T)$,
$u(t)=u_N(t)=u(t\wedge\tau_N)$ is the local solution of
(\ref{smain}). By the Theorem~\ref{thm local}, the following
energy equation holds:
\begin{eqnarray}\label{global}
&&e\big(u(t\wedge\tau_N)\big)=e(u_0)+
4\int_0^{t\wedge t_N}\int_D |u|^{p-2}uu_t(s)dxds-2\int_0^{t\wedge t_N}\int_D |u_t(s)|^qdxds\nonumber\\
  &&\quad\quad\quad\quad\ \ +2\int_0^{t\wedge t_N}\big(u_t(s),\varepsilon\sigma\big)dW_s+\varepsilon^2\int_0^{t\wedge t_N}
  \Big(\sigma(x,s)R e_i(x),\sigma(x,s)e_i(x)\Big)ds.
\end{eqnarray}
Using H\"{o}lder inequality and Young's inequality, we get
\begin{equation}\label{global1}
 \Big|\int_D |u|^{p-2}uu_t(s)dx\Big|\leq
 ||u||_{p}^{p-1}||u_t||_p\leq \beta ||u_t||_p^p+C_\beta
 ||u||_{p}^p,
\end{equation}
where $\beta>0$ and $C_\beta$ is a constant depending on $\beta$.
Since $q\geq p$ and (\ref{condition}), the embedding inequality
yields
\begin{equation}\label{global2}
 ||u_t||_p^p\leq C||u_t||_q^p,
\end{equation}
where $C$ is the embedding constant. Therefore, from
(\ref{global}), (\ref{global1}) and (\ref{global2}), we get
\begin{eqnarray}\label{global3}
&&e\big(u(t\wedge\tau_N)\big)\leq 4C\beta\int_0^{t\wedge t_N}
||u_t(s)||_q^pds-2\int_0^{t\wedge t_N} ||u_t(s)||_q^qdxds
+4C_\beta\int_0^{t\wedge t_N} ||u||_p^{p}ds\nonumber\\
  &&\quad\quad\quad\quad\ \ +e(u_0)+2\int_0^{t\wedge t_N}\big(u_t(s),\varepsilon\sigma\big)dW_s
  +\varepsilon^2c_0^2Tr R\int_0^{t\wedge t_N}
  ||\sigma(s)||_2^2ds.
\end{eqnarray}
Using $q\geq p$, at this point we distinguish two case:

$(i)$ Either $||u_t||_q^q>1$ so we choose $\beta$ small such that
$-2||u_t||_q^q+4C\beta||u_t||_q^p\leq0$.

$(ii)$ Or  $||u_t||_q^q\leq1$, in this case we have
$-2||u_t||_q^q+4C\beta||u_t||_q^p\leq 4C\beta$.\\
Therefore in either case, we have
\begin{eqnarray}\label{global4}
&&e\big(u(t\wedge\tau_N)\big)\leq e(u_0)+4C\beta(t\wedge t_N)
+4C_\beta\int_0^{t\wedge t_N} ||u||_p^{p}ds\nonumber\\
  &&\quad\quad\quad\quad\ \ +2\int_0^{t\wedge t_N}\big(u_t(s),\varepsilon\sigma\big)dW_s+
  \varepsilon^2c_0^2Tr R\int_0^{t\wedge t_N}
  ||\sigma(s)||_2^2ds.
\end{eqnarray}
By taking the expectation of (\ref{global4}), we obtain
\[
\textbf{E}e\big(u(t\wedge\tau_N)\big)\leq e(u_0)+4C\beta(t\wedge
t_N)+
  \varepsilon^2c_0^2Tr R\int_0^{t\wedge t_N}
  \textbf{E}||\sigma(s)||_2^2ds+K\int_0^{t\wedge
  t_N}\textbf{E}e\big(u(s)\big)ds,
\]
where $K>0$ is a constant, which, by the Gronwall inequality and
(\ref{mild}) , implies that
\begin{equation}\label{global5}
\textbf{E}e\big(u(T\wedge\tau_N)\big)\leq
\big(e(u_0)+CT\big)e^{KT}\leq C_T.
\end{equation}
On the other hand, we have
\[
\textbf{E}e\big(u(T\wedge\tau_N)\big)\geq \textbf{E}
\Big(I(\tau_N\leq T)e\big(u(\tau_N)\big)\Big)\geq
C\textbf{E}\Big(||u_{\tau_N}||^2_2I(\tau_N\leq T)\Big)\geq CN^2
P(\tau_N\leq T),
\]
where $I$ is the indicator function. In view of (\ref{global5}),
the above inequality gives
\[
P(\tau_\infty\leq T)\leq P(\tau_N\leq T)\leq \frac{C_T}{N^2},
\]
which, with the aid of the Borel-cantelli Lemma. implies that
\[
P(\tau_\infty\leq T)=0.
\]
or
\[
\lim_{N\rightarrow\infty} \tau_N=\infty\ \ \ a.s..
\]
Hence, on $[0,\tau_\infty\wedge T)=[0,T)$,
$u=\lim_{N\rightarrow\infty}u_N(t)$ is the global solution as
announced. Since $T>0$ was chosen arbitrarily, we may replace
$[0,T)$ by $[0,T]$.

To verify the energy bound (\ref{global0}), by the energy equation
(\ref{ienery}), (\ref{global1}),(\ref{global2}) and (\ref{mild}),
we have
\[
e\big(u(t)\big)\leq e(u_0)+(4C\beta+\varepsilon^2C_1)t
+4KC_\beta\int_0^{t} e\big(u(s)\big)ds
+2\int_0^{t}\big(u_t(s),\varepsilon\sigma\big)dW_s,
\]
where $C_1$ and $K$ are positive constants. The above inequality
yields
\begin{equation}\label{global6}
\textbf{E}\sup_{0\leq t\leq T}e\big(u(t)\big)\leq
e(u_0)+(4C\beta+\varepsilon^2C_1)T+4KC_\beta\int_0^{T}\textbf{E}\sup_{0\leq
s\leq T}e\big(u\big)ds+2\textbf{E}\sup_{0\leq t\leq
T}\int_0^{t}\big(u_t,\varepsilon\sigma\big)dW_s.
\end{equation}
By the Burkholder-Davis-Gundy inequality, we have
\begin{eqnarray}\label{global7}
&& \textbf{E}\sup_{0\leq t\leq
T}\Big|\int_0^t\big(u_t,\varepsilon\sigma\big)dW_s\big|\leq
C_2\textbf{E}\Big(\sup_{0\leq t\leq
T}||u_t||_2\Big(\varepsilon^2\sum_{i=1}^\infty\int^T_0\big(\sigma R e_i,\sigma e_i\big)dt\Big)^{\frac{1}{2}}\Big)\\
&&\leq \frac{1}{4}\textbf{E}\sup_{0\leq t\leq
T}||u_t||_2^2+C_3c_0^2\varepsilon^2TrR
\int^T_0\textbf{E}||\sigma(t)||_2^2dt
\end{eqnarray}
for some constant $C_2$, $C_3>0$. In view of (\ref{mild}),
(\ref{global6}) and (\ref{global7}), there exist positive
constants $C_4$ and $C_5$ depending on $\beta,\ T$ etc. such that
\[
\textbf{E}\sup_{0\leq t\leq T}e\big(u(t)\big)\leq
C_4+C_5\int_0^{T}\textbf{E}\sup_{0\leq s\leq T}e\big(u\big)ds.
\]
By applying the Gronwall inequality, the above gives
\[
\textbf{E}\sup_{0\leq t\leq T}e\big(u(t)\big)\leq C_4e^{C_5T},
\]
which implies the energy bound (\ref{global0}).
\end{proof}

\section{Explosive solution of (\ref{smain})}

In this section, we switch to discuss the explosion of the
solution to (\ref{smain}) for $p>q$. Throughout this section, we
suppose that $\sigma(x,t,\omega)\equiv \sigma(x,t)$ such that
\begin{equation}\label{blowup}
  \int_0^\infty \int_D \sigma^2(x,t)dxdt<\infty.
\end{equation}
As well-known, equation (\ref{smain}) is equivalent to the
following It\^{o} system
\begin{eqnarray}\label{emain}
  \begin{cases}
     du_{t}=v_tdt,\\
     dv_t=\Big(\Delta u_t- |v_t|^{q-2}v_t+|u_t|^{p-2}u_t\Big)dt+\varepsilon\sigma(x,t)d W(t,x),\\
    u_t(x,t)=0,         \ \ \ x\in\partial D,\\
    u_0(x,0)=u_{0}(x),\ \ \ v_0(x,0)=u_1(x),
\end{cases}
\end{eqnarray}
where $(u_0,u_1)\in H_0^1(D)\times L^2(D)$. Define energy
functional $E(t)$ associated to our system
\begin{equation*}
\mathcal{E}(t)=\frac{1}{2}||v_t(t)||^2_2+\frac{1}{2}||\nabla
u_t(t)||^2_2 -\frac{1}{p}||u_t||_{p}^{p}.
\end{equation*}
Before we state and prove our explosion result, we need the
following lemmas.

\begin{lem}\label{lem blowup}
Assume (\ref{condition}) and (\ref{blowup}) hold. Let $(u_t,v_t)$
be a solution of system (\ref{emain}) with initial data
$(u_0,u_1)\in H_0^1(D)\times L^2(D)$. Then we have
\begin{equation}\label{blowup3}
\frac{d}{dt}\textbf{E}\mathcal{E}(t)=-\textbf{E}||v_t||_q^q+\frac{1}{2}\varepsilon^2
\sum_{i=1}^\infty \int_D \lambda_ie_i^2(x)\sigma^2(x,t)dx,
\end{equation}
where $r(x,x)$ is defined in Section $2$, and
\begin{eqnarray}\label{blowup2}
&&\textbf{E}\big( u_t(t), v_t(t)\big)=\big(
u_0(x),v_0(x)\big)-\int_0^t \textbf{E}||\nabla
u_{s}||_2^2ds+\int_0^t\textbf{E}||v_s(s)||_2^2ds
\nonumber\\
&&\quad\quad\quad\quad\quad\quad\quad\
-\int_0^t\textbf{E}\big(|v_s|^{q-2}v_s, u_s\big)
ds+\int_0^t\textbf{E}||u_s(s)||_{p}^{p}ds,
\end{eqnarray}
\end{lem}
\begin{proof}
Using It\^{o} formula to $||v_t||_2^2$, we have
\begin{eqnarray}\label{blowup4}
&&||v_t||_2^2=||v_0||_2^2+2\int_0^t( v_s, d v_s)+\int_0^t( dv_s, dv_s)\nonumber\\
&&=||v_0||^2_2-2\int_0^t( \nabla u_s, \nabla v_s)
ds-2\int_0^t||v_s||_q^qds+2\int_0^t( v_s, |u_s|^{p-2}u_s)
ds\nonumber\\
&&\quad +2\int_0^t ( v_s, \varepsilon\sigma (x,s)dW(s))
+\varepsilon^2 \sum_{i=1}^\infty\int_0^t
\big(\sigma(x,s)Re_i,\sigma(x,s)Re_i\big)ds\nonumber\\
&&=2\mathcal{E}(0)-||\nabla
u_t(t)||^2_2-2\int_0^t||v_s||_q^qds +\frac{2}{p}||u_t(t)||_{p}^{p}\nonumber\\
&&\quad +2\int_0^t ( v_s, \varepsilon\sigma (x,s)dW(s))
+\varepsilon^2 \sum_{i=1}^\infty \int_0^t \int_D
\lambda_ie_i^2(x)\sigma^2(x,s)dxds.
\end{eqnarray}
(\ref{blowup3}) follows from (\ref{blowup4}) taking the
expectation and taking derivative. Next we turn to prove
(\ref{blowup2}).
\begin{eqnarray}\label{blowup5}
&&\big( u_t(t), v_t(t)\big)=\big( u_0, v_0\big)+\int_0^t\big( u_s(s),d v_s(s)\big)
+\int_0^t\big(v_s(s),d u_s(s)\big)\nonumber\\
&&=\big( u_0, v_0\big)-\int_0^t ||\nabla u_{s}(s)||^2_2ds-
\int_0^t\big( |v_s|^{q-2}v_s, u_s(s)\big) d\tau\nonumber\\
&&\quad+\int_0^t\big( u_s, |u_s|^{p-2}u_s\big) ds +\int_0^t \big(
u_s(s), \varepsilon\sigma (x,s)dW(s)\big) +\int_0^t
||v_s(s)||_2^2ds.
\end{eqnarray}
Then (\ref{blowup2}) follows from (\ref{blowup5}).
\end{proof}

Let
\[F(t)=\frac{1}{2}\varepsilon^2 \sum_{i=1}^\infty \int_0^t \int_D
\lambda_ie_i^2(x)\sigma^2(x,s)dxds.
\]
From (\ref{blowup}), we have
\begin{equation}\label{blowup0}
F(\infty)=\frac{1}{2}\varepsilon^2 \sum_{i=1}^\infty \int_0^\infty
\int_D \lambda_ie_i^2(x)\sigma^2(x,s)dxds\leq
\frac{1}{2}\varepsilon^2c^2_0Tr R \int_0^\infty
\int_D\sigma^2(x,s)dxds= E_1<\infty.
\end{equation}
Denote
\[
 H(t)=F(t)-\textbf{E}\mathcal{E}(t).
\]
Then, by (\ref{blowup3}), we get
\begin{equation}\label{blowup9'}
 H'(t)=F'(t)-\frac{d}{dt}\textbf{E}\mathcal{E}(t)=\textbf{E}||v_t||_q^q\geq0.
\end{equation}

\begin{lem}\label{lem blowup1}
Let $(u_t, u_t)$ is a solution of (\ref{emain}). Assume
(\ref{condition}) holds. Then there exists a positive constant
$C>1$ such that
\begin{equation}\label{blowup6}
\textbf{E}||u_t||_p^s\leq
C(F(t)-H(t)-\textbf{E}||v_t||^2_2+\textbf{E}||u_t||^p_p)
\end{equation}
for any $2\leq s\leq p$.
\end{lem}
\begin{proof}
If $||u_t||^p_p\leq1$ then $||u_t||^s_p\leq ||u_t||^2_p\leq
C||\nabla u_t||^2_2$ by Sobolev embedding. If $||u||^p_p\geq1$
then $||u_t||^s_p\leq ||u_t||^p_p$. Therefore, it follows that
\begin{equation}\label{blowup7}
 \textbf{E} ||u_t||^s_p\leq C(\textbf{E}||\nabla
 u_t||^2_2+\textbf{E}||u_t||^p_p).
\end{equation}
By the definition of energy function, we have
\begin{equation}\label{blowup7'}
\frac{1}{2}\textbf{E}||\nabla
 u_t||^2_2=\textbf{E}\mathcal{E}(t)-\frac{1}{2}\textbf{E}||v_t||^2_2+\frac{1}{p}\textbf{E}||u_t||^p_p
 =F(t)-H(t)-\frac{1}{2}\textbf{E}||v_t||^2_2+\frac{1}{p}\textbf{E}||u_t||^p_p.
\end{equation}
Then, (\ref{blowup6}) follows (\ref{blowup7}) and
(\ref{blowup7'}).
\end{proof}

 In the following, we switch
to discuss the explosion of the solution to (\ref{smain}) for
$p>q$. Actually, we have

\begin{thm}\label{thm blowup}
Assume (\ref{condition}) and (\ref{blowup}) hold.  Let $(u_t,v_t)$
be the solution of (\ref{emain}) with initial data $(u_0,u_1)\in
H_0^1(D)\times L^2(D)$ satisfying
\begin{equation}\label{blowup8}
\mathcal{E}(0)\leq -(1+\beta)E_1,
\end{equation}
where $\beta>0$ is any constant and $E_1$ is defined in
(\ref{blowup0}). If $p>q$, then the solution $(u_t,v_t)$ and the
lifespan $\tau_\infty$ defined in Section $3$
with $L^2$ norm, either\\
(1) $\textbf{P}(\tau_\infty<\infty)>0$, i.e., $u_t(t)$ in $L^2$
norm blows up in finite time with positive probability, or\\
(2) there existence a positive time $T^*\in(0,T_0]$ such that
\[
 \lim_{t\rightarrow T^*}\textbf{E}\mathcal{E}(t)=+\infty.
\]
with
\[
 T_0= \frac{1-\alpha}{\alpha K \mathcal{E}^\frac{\alpha}{1-\alpha}(0)},
\]
where $\alpha$, $K$ are given later.
\end{thm}
\begin{proof}
For the lifespan $\tau_\infty$ of the solution $\{u_t(t);\
t\geq0\}$ of (\ref{smain}) with $L^2$ norm, let us consider the
case when $\textbf{P}(\tau_\infty=+\infty)=1$. Then, for
sufficiently large $T>0$, by (\ref{blowup9'}) and (\ref{blowup8}),
we have
\begin{equation}\label{blowup9}
0< (1+\beta)E_1\leq  -\mathcal{E}(0)=H(0)\leq H(t)\leq
F(t)+\frac{1}{p}\textbf{E}||u||^p_p\leq
E_1+\frac{1}{p}\textbf{E}||u||^p_p.
\end{equation}
Define by
\[
 L(t):=H^{1-\alpha}(t)+\mu\textbf{E}(u_t,v_t),
\]
for small $\mu$ to be chosen later and for
\begin{equation}\label{blowup10}
0<\alpha<\min\Big\{\frac{1}{2},\frac{p-q}{pq}\Big\}.
\end{equation}
Taking a derivative of $L(t)$ and using (\ref{blowup2}) and
(\ref{blowup9'}), we obtain
\begin{eqnarray}\label{blowup11}
&&L'(t)=(1-\alpha)H^{-\alpha}(t)H'(t)+\mu\Big(-\textbf{E}||\nabla u_t||_2^2-\textbf{E}(|v_t|^{q-2}v_t,u_t)
+\textbf{E}||u_t||^p_p+\textbf{E}||v_t||^2_2\Big)\nonumber\\
&&\quad\quad\ =(1-\alpha)H^{-\alpha}(t)\textbf{E}||v_t||_q^q+\mu p
H(t)+\mu(\frac{p}{2}+1)\textbf{E}||v_t||^2_2\nonumber\\
&&\quad\quad\ \ \ +\mu(\frac{p}{2}-1)\textbf{E}||\nabla
u_t||_2^2-\mu\textbf{E}(|v_t|^{q-2}v_t,u_t)-\mu p F(t).
\end{eqnarray}
Exploiting the inequality $\textbf{E}||u_t||_q^q\leq C
\textbf{E}||u_t||_p^q$ and the assumption $q<p$, we obtain
\begin{eqnarray}\label{blowup11'}
&&\Big|\textbf{E}(|v_t|^{q-2}v_t,u_t)\Big|\leq\big(\textbf{E}||v_t||^q_q\big)^{\frac{q-1}{q}}
\big(\textbf{E}||u_t||^q_q\big)^{\frac{1}{q}}\leq
C\big(\textbf{E}||v_t||^q_q\big)^{\frac{q-1}{q}}\big(\textbf{E}||u_t||^q_p\big)^{\frac{1}{q}}\nonumber\\
 &&\leq
C\big(\textbf{E}||v_t||^q_q\big)^{\frac{q-1}{q}}\big(\textbf{E}||u_t||^p_p\big)^{\frac{1}{p}}\leq
C\big(\textbf{E}||v_t||^q_q\big)^{\frac{q-1}{q}}\big(\textbf{E}||u_t||^p_p\big)^{\frac{1}{q}}
\big(\textbf{E}||u_t||^p_p\big)^{\frac{1}{p}-\frac{1}{q}}.
\end{eqnarray}
The Young's inequality gives
\begin{equation}\label{blowup10'}
\big(\textbf{E}||v_t||^q_q\big)^{\frac{q-1}{q}}\big(\textbf{E}||u_t||^p_p\big)^{\frac{1}{q}}\leq
\frac{q-1}{q}k\textbf{E}||v_t||^q_q+\frac{k^{1-q}}{q}\textbf{E}||u_t||^p_p.
\end{equation}
In view of (\ref{blowup9}), we get
\[
 \textbf{E}||u_t||^p_p\geq p\big(H(t)-F(t)\big)\geq \kappa H(t),
\]
where $\kappa=p\beta/(1+\beta)$. We choose $\alpha$ satisfying
(\ref{blowup10}) and assume $H(0)>1$, we have
\begin{equation}\label{blowup01}
 \big(\textbf{E}||u_t||^p_p\big)^{\frac{1}{p}-\frac{1}{q}}\leq
 \kappa^{\frac{1}{p}-\frac{1}{q}} H^{\frac{1}{p}-\frac{1}{q}}(t)
 \leq \kappa^{\frac{1}{p}-\frac{1}{q}}H^{-\alpha}(t) \leq \kappa^{\frac{1}{p}-\frac{1}{q}}H^{-\alpha}(0).
\end{equation}
Substituting (\ref{blowup10'}) and (\ref{blowup01}) into
(\ref{blowup11'}), we obtain
\begin{equation}\label{blowup02}
 \Big|\textbf{E}(|v_t|^{q-2}v_t,u_t)\Big|\leq C_1
 \frac{q-1}{q}k\textbf{E}||v_t||^q_q H^{-\alpha}(t)+C_1\frac{k^{1-q}}{q}\textbf{E}||u_t||^p_p
 H^{-\alpha}(0),
\end{equation}
where $C_1=C\kappa^{\frac{1}{p}-\frac{1}{q}}$. Thus, from
(\ref{blowup11}) and (\ref{blowup02}) it follow that
\begin{eqnarray}\label{blowup12}
&&L'(t)\geq\Big((1-\alpha)-C_1\frac{q-1}{q}\mu
k\Big)H^{-\alpha}(t)\textbf{E}||v_t||_q^q+\mu p
H(t)+\mu(\frac{p}{2}+1)\textbf{E}||v_t||^2_2-\mu p F(t)\nonumber\\
&&\quad\quad\ \ \ +\mu(\frac{p}{2}-1)\textbf{E}||\nabla u_t||_2^2
-\mu C_1\frac{k^{1-q}}{q}H^{-\alpha}(0) \textbf{E}||u_t||^p_p.
\end{eqnarray}
We now use Lemma \ref{lem blowup1} with $s=p$ to deduce from
(\ref{blowup12})
\begin{eqnarray}\label{blowup15}
&&L'(t)\geq\Big((1-\alpha)-C_1\frac{q-1}{q}\mu
k\Big)H^{-\alpha}(t)\textbf{E}||v_t||_q^q+\mu p
H(t)+\mu(\frac{p}{2}+1)\textbf{E}||v_t||^2_2-\mu p F(t)\nonumber\\
&&\quad\quad\ \ \ +\mu(\frac{p}{2}-1)\textbf{E}||\nabla u_t||_2^2
-\mu k^{1-q}C_2\Big(
F(t)-H(t)-\textbf{E}||v_t||^2_2+\textbf{E}||u_t||^p_p\Big)\nonumber\\
&&\quad\quad \ \geq \Big((1-\alpha)-C_1\frac{q-1}{q}\mu
k\Big)H^{-\alpha}(t)\textbf{E}||v_t||_q^q+\mu(\frac{p}{2}+1+k^{1-q}C_2)\textbf{E}||v_t||^2_2
+\mu(\frac{p}{2}-1)\textbf{E}||\nabla u_t||_2^2\nonumber\\
&&\quad\quad\ \ \ +\mu(p+k^{1-q}C_2)H(t)-\mu
k^{1-q}C_2\textbf{E}||u_t||^p_p-\mu(p+k^{1-q}C_2)F(t),
\end{eqnarray}
where $C_2=C_1H^{-\alpha}(0)/q$. Noting that
\[
 H(t)=F(t)+\frac{1}{p}\textbf{E}||u_t||^p_p-\frac{1}{2}\textbf{E}||\nabla u_t||^2_2-\frac{1}{2}\textbf{E}||v_t||^2_2
\]
and writing $p=2C_3+(p-2C_3)$, where $C_3<(p-2)/2$, the estimate
(\ref{blowup15}) implies
\begin{eqnarray}\label{blowup16}
&&L'(t)\geq \Big((1-\alpha)-C_1\frac{q-1}{q}\mu
k\Big)H^{-\alpha}(t)\textbf{E}||v_t||_q^q+\mu(\frac{p}{2}+1+k^{1-q}C_2-C_3)\textbf{E}||v_t||^2_2\nonumber\\
&&\quad\quad\ \ \ +\mu(\frac{p}{2}-1-C_3)\textbf{E}||\nabla u_t||_2^2+\mu(p-2C_3+k^{1-q}C_2)H(t)\nonumber\\
&&\quad\quad\ \ \ +\mu
\big(\frac{2C_3}{p}-k^{1-q}C_2\big)\textbf{E}||u_t||^p_p-\mu(p-2C_3+k^{1-q}C_2)F(t).
\end{eqnarray}
In view of (\ref{blowup8}) and (\ref{blowup9}), we get
\[
 (p-2C_3+k^{1-q}C_2)F(t)\leq (p-2C_3+k^{1-q}C_2)E_1\leq
 \frac{(p-2C_3+k^{1-q}C_2)}{1+\beta}H(t).
\]
Substituting the above inequality into (\ref{blowup16}), we get
\begin{eqnarray*}
&&L'(t)\geq \Big((1-\alpha)-C_1\frac{q-1}{q}\mu
k\Big)H^{-\alpha}(t)\textbf{E}||v_t||_q^q+\mu\Big((\frac{p}{2}+1+k^{1-q}C_2-C_3)\textbf{E}||v_t||^2_2\nonumber\\
&&\quad\quad\  +(\frac{p}{2}-1-C_3)\textbf{E}||\nabla
u_t||_2^2+(p-2C_3+k^{1-q}C_2)\frac{\beta}{1+\beta}H(t) +
\big(\frac{2C_3}{p}-k^{1-q}C_2\big)\textbf{E}||u_t||^p_p\Big).
\end{eqnarray*}
At this point, we choose $k$ large enough so that the above
inequality becomes
\begin{equation}\label{blowup17}
L'(t)\geq \Big((1-\alpha)-C_1\frac{q-1}{q}\mu
k\Big)H^{-\alpha}(t)\textbf{E}||v_t||_q^q+\mu\gamma\big(H(t)+\textbf{E}||\nabla
u_t||_2^2+\textbf{E}||v_t||^2_2+\textbf{E}||u_t||^p_p\big),
\end{equation}
where $\gamma>0$ is the minimum of the coefficients of $H(t)$,
$\textbf{E}||\nabla u_t||_2^2,\ \textbf{E}||v_t||^2_2,\
\textbf{E}||u_t||^p_p$ in (\ref{blowup17}). Once $k$ is fixed, we
pick $\mu$ small enough so that
\[
(1-\alpha)-C_1\frac{q-1}{q}\mu k\geq0
\]
and
\[
L(0)=H^{1-\alpha}(0)+\mu (u_0,u_1)>0.
\]
Therefore, (\ref{blowup17}) takes on the form
\begin{equation}\label{blowup18}
L'(t)\geq \mu\gamma\big(H(t)+\textbf{E}||\nabla
u_t||_2^2+\textbf{E}||v_t||^2_2+\textbf{E}||u_t||^p_p\big)\geq0.
\end{equation}
Consequently, we have
\[
L(t)\geq L(0)>0,\ \ \ \forall t\geq0.
\]

By H\"{o}lder inequality, we get
\[
\Big|\textbf{E}\big(u_t,v_t\big)\Big|\leq
\big(\textbf{E}||u_t||_2^2\big)^{\frac{1}{2}}\big(\textbf{E}||v_t||_2^2\big)^{\frac{1}{2}}
\leq
C\big(\textbf{E}||u_t||_p^2\big)^{\frac{1}{2}}\big(\textbf{E}||v_t||_2^2\big)^{\frac{1}{2}},
\]
which, by young's inequality implies
\begin{eqnarray}\label{blowup19}
&&\Big|\textbf{E}\big(u_t,v_t\big)\Big|^{\frac{1}{1-\alpha}}
 \leq C \big(\textbf{E}||u_t||_p^2\big)^{\frac{1}{2(1-\alpha)}}
 \big(\textbf{E}||v_t||_2^2\big)^{\frac{1}{2(1-\alpha)}}\nonumber\\
&&\quad\quad\quad\quad\quad\ \ \ \ \leq
C\Big(\big(\textbf{E}||u_t||_p^2\big)^{\frac{\theta}{2(1-\alpha)}}+
\big(\textbf{E}||v_t||_2^2\big)^{\frac{\eta}{2(1-\alpha)}}\Big),
\end{eqnarray}
for $1/\theta+1/\eta=1$. We take $\eta=2(1-\alpha)$. Then, by
(\ref{blowup10}),
\[
\frac{\theta}{2(1-\alpha)}=\frac{1}{1-2\alpha}=\frac{pq}{pq-2p+2q}\leq\frac{p}{2},
\]
i.e., $2/(1-2\alpha)\leq p$. Using $\alpha<1/2$, (\ref{blowup19})
becomes
\[
\Big|\textbf{E}\big(u_t,v_t\big)\Big|^{\frac{1}{1-\alpha}}\leq
C\Big(\big(\textbf{E}||u_t||_p^2\big)^{\frac{1}{1-2\alpha}}+
\textbf{E}||v_t||_2^2\Big)\leq
C\Big(\textbf{E}||u_t||_p^{\frac{2}{1-2\alpha}}+
\textbf{E}||v_t||_2^2\Big).
\]
Using Lemma \ref{lem blowup1} with $s=2/(1-2\alpha)$, we get
\begin{equation}\label{blowup20}
\Big|\textbf{E}\big(u_t,v_t\big)\Big|^{\frac{1}{1-\alpha}}\leq
C\big(H(t)+\textbf{E}||\nabla
u_t||_2^2+\textbf{E}||v_t||^2_2+\textbf{E}||u_t||^p_p\big), \ \ \
\forall t\geq0.
\end{equation}
Therefore, we have
\begin{eqnarray}\label{blowup21}
&&L^{\frac{1}{1-\alpha}}(t)=\Big(H^{1-\alpha}(t)+\mu\textbf{E}\big(u_t,v_t\big)\Big)^{\frac{1}{1-\alpha}}
 \leq 2^{\frac{1}{1-\alpha}} \Big(H(t)+\mu\Big|\textbf{E}\big(u_t,v_t\big)\Big|^{\frac{1}{1-\alpha}}\Big)\nonumber\\
&&\quad\ \ \ \ \ \leq C\big(H(t)+\textbf{E}||\nabla
u_t||_2^2+\textbf{E}||v_t||^2_2+\textbf{E}||u_t||^p_p\big), \ \ \
\forall t\geq0.
\end{eqnarray}
Combining (\ref{blowup18}) and (\ref{blowup21}), we obtain
\begin{equation}\label{blowup22}
L'(t)\geq K L^{\frac{1}{1-\alpha}}, \ \ \ \forall t\geq0,
\end{equation}
where $K$ is a positive constant. A simple integration of
(\ref{blowup22}) over $(0,t)$ then yields
\begin{equation}\label{blowup23}
L^{\frac{\alpha}{1-\alpha}}(t)\geq
\frac{1-\alpha}{(1-\alpha)L^{-\frac{\alpha}{1-\alpha}}(0)-\alpha
Kt}.
\end{equation}
Let
\[
 T_0= \frac{1-\alpha}{\alpha K \mathcal{E}^\frac{\alpha}{1-\alpha}(0)}.
\]
Then $L(t)\rightarrow+\infty$ as $t\rightarrow T_0$. This means
that there exists a positive time $T^*\in(0,T_0]$ such that
\[
 \lim_{t\rightarrow T^*}\textbf{E}\mathcal{E}(t)=+\infty.
\]

As for the case when $\textbf{P}(\tau_\infty=+\infty)<1$ (i.e.,
$\textbf{P}(\tau_\infty<+\infty)>0$), then $u_t(t)$ in $L^2$ norm
blows up in finite time interval $[0,\tau_\infty]$ with positive
probability.
\end{proof}

\begin{rem}
In the classical (deterministic) case of $\varepsilon=0$, it is
well known that for $(u_0,v_0)\in H^1_0(D)\times L^2(D)$, the
condition $\mathcal{E}(0)\leq0$ already imply finite-time blowup
of (\ref{smain}) (see e.g. \cite{M1}). If $\varepsilon>0$, by our
results, to balance the influence of $W(t,x)$ such that the local
solution of (\ref{smain}) is blow-up with positive probability or
explosive in $L^2$ sense, the initial energy should be satisfied
$\mathcal{E}(0)\leq -\frac{1}{2}(1+\beta)\varepsilon^2
r_0^2\int_0^\infty \int_D \sigma^2(x,t)dxdt$.
\end{rem}



\begin{thebibliography}{99}
\bibitem{HZ} A. Haraux, E. Zuazua,
             \textit{Decay estimates for some semilinear damped hyperbolic problems},
             Arch. Ration. Mech. Anal. \textbf{150} (1988), 191--206.
\bibitem{K} M. Kopackova,
             \textit{Remarks on bounded solutions of a semilinear dissipative hyperbolic equation},
             Comment. Math. Univ. Carolin. \textbf{30}(4) (1989), 713--719.
\bibitem{B} J. Ball,
             \textit{Remarks on blow up and nonexistence theorems for nonlinear evolutions equations},
             Quart. J. Math. Oxford \textbf{28} (2) (1977), 473--486.
\bibitem{KL} V.K. Kalantarov, O.A. Ladyzhenskaya,
             \textit{The occurrence of collapse for quasilinear equations of parabolic and hyperbolic type},
               J. Soviet Math. \textbf{10} (1978), 53--70.
\bibitem{L} H.A. Levine,
             \textit{Instability and nonexistence of global solutions to nonlinear wave equations of the form},
              Trans. Amer. Math. Soc. \textbf{192} (1974), 1--21.
\bibitem{L1} H.A. Levine,
             \textit{Some additional remarks on the nonexistence of global solutions to nonlinear wave equations},
             SIAM J. Math. Anal. \textbf{5} (1974), 138--146.
\bibitem{GT} V. Georgiev, G. Todorova,
             \textit{Existence of a solution of the wave equation with nonlinear damping and source term},
              J. Differential Equations \textbf{109} (1994), 295--308.
\bibitem{M1} S.A. Messaoudi,
             \textit{Blow up in a nonlinearly damped wave equation},
              Math. Nachr. \textbf{231} (2001), 1--7.
\bibitem{LS} H.A. Levine, J. Serrin,
             \textit{Global nonexistence theorems for quasilinear evolution equation with dissipation},
              Arch. Ration. Mech. Anal. \textbf{137} (1997), 341--361.
\bibitem{LR} H.A. Levine, S. Ro Park,
             \textit{Global existence and global nonexistence of solutions of the Cauchy problem for a nonlinearly damped wave equation},
             J. Math. Anal. Appl. \textbf{228} (1998), 181--205.
\bibitem{V}  E. Vitillaro,
            \textit{Global nonexistence theorems for a class of evolution equations with dissipation},
            Arch. Ration. Mech. Anal. \textbf{149} (1999), 155--182.
\bibitem{MS} S.A. Messaoudi, B. Said-Houari,
             \textit{Blow up of solutions of a class of wave equations with nonlinear damping and source terms},
              Math. Methods Appl. Sci. \textbf{27} (2004), 1687--1696.
\bibitem{C} P.L. Chow,
             \textit{Stochastic wave equations with polynomial nonlinearity},
              Ann. Appl. Probab. \textbf{12} (2002), 361--381.
\bibitem{C3} P.L. Chow,
             \textit{Nonlinear stochstic wave equations: blow-up of second moments in $L^2$-norm},
              Ann. Appl. Probab. \textbf{19} (2009), 2039--2046.
\bibitem{BTW} L.J. Bo, D. Tang, Y.G. Wang,
             \textit{Explosive solutions of stochastic wave equations with damping on $\mathbb{R}^d$},
              J. Differential Equations  \textbf{244} (2008), 170--187.
\bibitem{C1} P.L. Chow,
             \textit{Asymptotics of solutions to semilinear stochastic wave equations},
              Ann. Appl. Probab. \textbf{16} (2006), 757--789.
\bibitem{C2} P.L. Chow,
             \textit{Asymptotic solutions of a nonlinear stochastic beam equation},
              Discrete Contin. Dyn. Syst. Ser. B. \textbf{6} (2006), 735--749.
\bibitem{BM} Z. Brze\'{z}niak, B. Maslowski, J. Seidler,
             \textit{Stochastic nonlinear beam equations},
              Probab. Theory Related Fields  \textbf{132} (2005), 119--149.
\bibitem{CN} R. Carmona, D. Nualart,
             \textit{Random non-linear wave equation: Smoothness of solutions},
              Probab. Theory Related Fields   \textbf{95} (1993), 87--102.
\bibitem{HC} H. Crauel, A. Debussche, F. Flandoli,
             \textit{Random attractors},
              J. Dynam. Differential Equations \textbf{9} (1997), 307--341.
\bibitem{DF} R. Dalang, N. Frangos,
             \textit{The stochastic wave equation in two spatial dimensions},
              Ann. Probab. \textbf{26} (I) (1998), 187--212.
\bibitem{MM} A. Millet, P.L. Morien,
             \textit{On a nonlinear stochastic wave equation in the plane: Existence and uniqueness of the solution},
              Ann. Appl. Probab. \textbf{11} (2001), 922--951.
\bibitem{EP} E. Pardoux,
             \textit{Equations aux deriv\'{e}es partielles stochastiques nonlin\'{e}aries monotones},
              Th\`{e}se, Universit\'{e} Paris XI 1975.
\bibitem{JUK} J.U. Kim,
             \textit{On the stochastic wave equation with nonlinear damping},
              Appl. Math. Optim.  \textbf{58} (2008), 29--67.
\bibitem{BPT} V. Barbu, G.D. Prato, L. Tubaro,
             \textit{Stochastic wave equations with dissipative damping},
              Stochastic Process. Appl. \textbf{117} (2007), 1001--1013.
\bibitem{PZ} G. Da Prato, J. Zabczyk,
             \textit{Stochastic Equations in Infinite Dimensions, Cambridge University Press},
              Cambridge, 1992.
\bibitem{AP} A. Pazy,
             \textit{Semigroups of linear operators and applications to partial differential equations},
             Springer-verlag, New York, 1983.
\bibitem{LL} J.L. Lions
             \textit{Quelques m\'{e}thodes de r\'{e}solution des probl\'{e}mes aux limites non lin\'{e}aires},
             Dunod, Paris (1969).
\end{thebibliography}
\end{document}